\documentclass[11pt,reqno]{amsart}
\usepackage{amsthm,amssymb,amsfonts}
\usepackage[nosumlimits]{mathtools}
\usepackage{mathrsfs}
\usepackage{graphicx}
\usepackage{accents}
\usepackage{algorithm}
\usepackage{algpseudocode}
\usepackage{tikz-cd}
\usepackage{enumerate}
\usepackage{quiver}
\usepackage{stmaryrd}
\usepackage{enumitem}
\usepackage{subcaption}
\usepackage{adjustbox}
\usepackage[margin=1in]{geometry}

\usepackage{hyperref,xcolor}
\hypersetup{
    colorlinks,
    linkcolor={red!50!black},
    citecolor={blue!50!black},
    urlcolor={blue!80!black}
}

\usepackage{thmtools}
\usepackage{thm-restate}

\let\oldlhd\unlhd
\renewcommand*{\unlhd}{\mathrel{\mskip0.1mu \oldlhd \mskip0.1mu}}

\newcommand{\myto}{\mathrel{\tikz[baseline]\draw[->, line width=0.06em](0,0.35em)--(1.5em,0.35em);}}
\newcommand{\mydashrightarrow}{\mathrel{\tikz[baseline, line width=0.06em]\draw[dashed,->](0,0.35em)--(1.5em,0.35em);}}

\theoremstyle{plain}
\newtheorem{theorem}{Theorem}[section]
\newtheorem{proposition}[theorem]{Proposition}
\newtheorem{lemma}[theorem]{Lemma}
\newtheorem{corollary}[theorem]{Corollary}
\newtheorem{remark}[theorem]{Remark}
\theoremstyle{definition}
\newtheorem{definition}[theorem]{Definition}
\newtheorem{example}[theorem]{Example}

\newtheorem{conjecture}[theorem]{Conjecture}

\DeclareMathOperator{\adj}{adj}
\DeclareMathOperator{\AR}{AR}
\DeclareMathOperator{\GR}{GR}
\DeclareMathOperator{\Id}{Id}

\DeclareMathOperator{\codim}{codim}
\DeclareMathOperator{\spa}{span}
\DeclareMathOperator{\Hom}{Hom}

\DeclareMathOperator{\PR}{PR}

\DeclareMathOperator{\height}{ht}
\DeclareMathOperator{\str}{str}
\DeclareMathOperator{\Brk}{Brk}

\DeclareMathOperator{\rank}{rank}

\DeclareMathOperator{\Spec}{Spec}
\DeclareMathOperator{\ch}{char}

\DeclareMathOperator{\Sing}{Sing}

\DeclareMathOperator{\mdeg}{m-deg}

\begin{document}
\title{Geometry of multilinear varieties over infinite fields and its applications}
\date{}
\author[Q-Y.~Chen]{Qiyuan~Chen}
\address{State Key Laboratory of Mathematical Sciences, Academy of Mathematics and Systems Science, Chinese Academy of Sciences, Beijing 100190, China}
\email{chenqiyuan@amss.ac.cn}
\author[K.~Ye]{Ke Ye}
\address{State Key Laboratory of Mathematical Sciences, Academy of Mathematics and Systems Science, Chinese Academy of Sciences, Beijing 100190, China}
\email{keyk@amss.ac.cn}
\begin{abstract}
Multilinear varieties,  defined as the sets of rational points of varieties cut out by multilinear functions, were first introduced and studied by Gowers and Milićević [Proc. Edinb. Math. Soc., 2021] for finite $\mathbb{K}$. In this paper, we investigate multilinear varieties over infinite fields from a geometric perspective.  We establish two fundamental results: a codimension formula for the Zariski closure of a multilinear variety, and the existence of a high-dimensional irreducible subvariety passing through any given $\mathbb{K}$-rational point. These results serve as a geometric foundation for analyzing various ranks of tensors and homogeneous polynomials, including partition rank, analytic rank, geometric rank,  (collective) strength and (collective) Birch rank.  As applications,  we resolve the Adiprasito-Kazhdan-Ziegler conjecture [arXiv:2102.03659,  2021] on the stability of partition rank for perfect infinite fields.  We thereby settle the stability conjecture for collective strength [Selecta Math., 2024],  as well as the conjecture on the linear equivalence between strength and Birch rank [arXiv:2410.00248, 2024] for such fields.  Moreover,  our results immediately yield a strengthening of the theorems of Bik-Draisma-Snowden [arXiv:2401.02067, 2024] and Lampert-Snowden [arXiv:2406.18498, 2024], for multilinear varieties over infinite fields. 
\end{abstract}
\maketitle


\section{Introduction}
Let $\mathbb{K}$ be a field and let $f:\mathbb{K}^{n_1} \times \cdots \times \mathbb{K}^{n_d} \to \mathbb{K}$ be a function.  For each $j \in [d]$,  we denote by $\deg_{j} f = e$ if $f(\mathsf{v}_1,\dots, \mathsf{v}_{j-1},x_j,  \mathsf{v}_{j+1},\dots,  \mathsf{v}_{d}): \mathbb{K}^{n_j} \to \mathbb{K}$ is a homogeneous polynomial of degree $e$ in $x_j$,  for any $\mathsf{v}_s \in \mathbb{K}^{n_s}$,  $s \in [d] \setminus \{j\}$.  If $\deg_j f \le e_j$ for each $j \in [k]$,  then we write $\mdeg f \le (e_1,\dots,  e_d)$.  Given functions $f_1,\dots,  f_m: \mathbb{K}^{n_1} \times \cdots \times \mathbb{K}^{n_d} \to \mathbb{K}$ with $\mdeg f_j \le (1,\dots,  1)$,  $j \in [m]$,  the \emph{multilinear variety} $Z_{f_1,\dots,  f_m} (\mathbb{K})$ consists of $\mathbb{K}$-rational points of $Z_{f_1,\dots,  f_m} \coloneqq \{\mathsf{v} \in \overline{\mathbb{K}}^{n_1} \times \cdots \times \overline{\mathbb{K}}^{n_d}:  f_1 (\mathsf{v}) = \cdots = f_m(\mathsf{v}) = 0\}$. Given a multilinear map $F = (f_1,\dots,  f_m) \in \Hom(\mathbb{K}^{n_1} \times \cdots \times \mathbb{K}^{n_d},  \mathbb{K}^m)$,  we denote $Z_F(\mathbb{K}) \coloneqq Z_{f_1,\dots,  f_m} (\mathbb{K})$.

Multilinear varieties were first defined by Gowers and Mili\'{c}evi\'{c} in \cite{GM21} over finite fields to study the extension problem of multilinear maps.  Recently,  Mili\'{c}evi\'{c} proved in \cite{Luka25} that every dense multilinear variety over a finite field must contain a multilinear variety of small codimension.  It is worth mentioning that a multilinear variety having codimension at most $r$ in \cite{Luka25} means it is cut out by at most $r$ multilinear functions,  which differs from the usual codimension in algebraic geometry.  On the other hand,  multilinear varieties of the form $Z_F(\mathbb{K})$ are ubiquitous in the study of various structures of multilinear maps.  For instance,  the geometric rank \cite{kopparty2020geometric},  isotropy index \cite{Baer38,BGH87,CXY26}, completeness index \cite{Qiao23,CY25},  together with the analytic rank over finite fields \cite{gowers2011linear} and infinite fields \cite{kazhdan2024schmidt,bik2025strength},  are all defined as invariants of $Z_F(\mathbb{K})$.  

Suppose $\mathbb{K}$ is infinite.  Given a subset $S \subseteq \mathbb{K}^n$,  we denote $\mathcal{I}(S) \coloneqq \{g \in \mathbb{K}[x_1,\dots,  x_n]: g(\mathsf{v}) = 0,\; \mathsf{v} \in S\}$,  and
\begin{equation}\label{eq:Xclosure}
\overline{S} \coloneqq \{ v\in \overline{\mathbb{K}}^n: f(v)=0,\; f\in \mathcal{I}(S) \}.
\end{equation} 
Note that $\overline{S}$ is the closure of $S$ in $\overline{\mathbb{K}}^n$,  not the closure of $S$ in $\mathbb{K}^n$. The objects of study in this paper are $Z_{f_1,\dots,  f_m} (\mathbb{K})$ and $\overline{Z_{f_1,\dots,  f_m}(\mathbb{K})}$.  

\subsection{Main results}\label{subsec:main}
We establish in Section~\ref{sec:geometry} two geometric properties of multilinear varieties that distinguish them from general varieties.  The first is a codimension formula for $\overline{Z_F(\mathbb{K})}$,  proved in Subsection~\ref{subsec:codim}. 
\begin{theorem}[Codimension formula]\label{thm:codim}
For any infinite field $\mathbb{K}$ and multilinear map $F: \mathbb{K}^{n_1} \times \cdots \times \mathbb{K}^{n_d} \to \mathbb{K}^m$,  we have 
\begin{equation}\label{thm:codim:eq}
\codim \overline{Z_F(\mathbb{K})} = \min_{0 \le r \le \min\{n_d,  m\} } \big\lbrace
r + \codim \overline{W_{r,F}(\mathbb{K})}
\big\rbrace,
\end{equation}
where $W_{r,F}(\mathbb{K})$ is the set defined as
\[
W_{r,F}(\mathbb{K}) \coloneqq \{ (\mathsf{v}_1,\dots,  \mathsf{v}_{d-1}) \in \mathbb{K}^{n_1} \times \cdots \times \mathbb{K}^{n_{d-1}}:  \rank F(\mathsf{v}_1,\dots,  \mathsf{v}_{d-1},\cdot) = r \}. 
\]
\end{theorem}
It is important to note that $\codim \overline{Z_F(\mathbb{K})}$ differs from $\height \mathcal{I}(Z_F(\mathbb{K}))$.  In this regard,  a remark follows.  Assume that $X \subseteq \mathbb{K}^n$ is the set of common zeros of $P_1,\dots,  P_m \in \mathbb{K}[x_1,\dots,  x_n]$.  It is clear that $\mathcal{I} \coloneqq \langle P_1,\dots,  P_m \rangle \subseteq \mathcal{I}(X)$ and 
\begin{equation}\label{eq:codim-ht}
\codim \overline{X} = \height \mathcal{I}(X) \ge \height \mathcal{I}.
\end{equation}
We notice that the inequality in \eqref{eq:codim-ht} can be strict in general. For instance,  if $P(x,y) = x^2 + y^2 \in \mathbb{R}[x,y]$,  then $X = \overline{X} = \{ (0,0) \}$ and $\mathcal{I}_X = \langle x, y \rangle$.  This implies $\codim \overline{X} = 2 > \height \langle P \rangle = 1$.  Indeed,  $\height \mathcal{I}$ equals the codimension of the affine scheme $\Spec(\mathbb{K}[x_1,\dots,  x_n]/\mathcal{I})$,  which is extensively studied in classical algebraic geometry \cite{hartshorne2013algebraic,Eisenbud00,Vakil25}. Moreover,  we show in Example~\ref{example1} that the inequality in \eqref{eq:codim-ht} can also be strict for multilinear varieties. 

The second result concerns the existence of a high-dimensional subvariety of $\overline{Z_{f_1,\dots,  f_m}(\mathbb{K})}$,  passing through a given point $\mathsf{v} \in Z_{f_1,\dots,  f_m}(\mathbb{K})$. Indeed,  we prove in Subsection~\ref{subsec:subvar} that such a subvariety exists in a more general setting. 
\begin{theorem}[High-dimensional subvarieties]\label{thm:Krull}
Let $\mathbb{K}$ be an infinite field. Suppose that $f_1,\dots,  f_m: \mathbb{K}^{n_1} \times \cdots \times \mathbb{K}^{n_d} \to \mathbb{K}$ are functions such that $\mdeg f_1,  \dots,  \mdeg f_m \le (1,\dots, 1)$. For any $c_1,\dots,  c_m \in \mathbb{K}$ and any $\mathsf{v} \in  \mathbb{K}^{n_1} \times \cdots \times \mathbb{K}^{n_d} $ such that $(f_1(\mathsf{v}),  \dots,  f_m(\mathsf{v})) = (c_1,\dots,  c_m)$,  there exists some $W \subseteq Z \coloneqq \lbrace \mathsf{v} \in \mathbb{K}^{n_1} \times \cdots \times \mathbb{K}^{n_d}: (f_1(\mathsf{v}),\dots,  f_m(\mathsf{v})) = (c_1,\dots,  c_m)  \rbrace$ satisfying the following properties: 
\begin{enumerate}[label = (\alph*)]
\item $\mathsf{v}$ is contained in $\overline{W}$.  \label{thm:Krull:cond1}
\item $\overline{W}$ is irreducible.  \label{thm:Krull:cond2}
\item $\codim (\overline{W}) \le 2^d m$. \label{thm:Krull:cond3}
\end{enumerate}
\end{theorem}
Theorem~\ref{thm:Krull} can fail when $f_1$,  $\dots$,  $f_m$ are not multilinear.  As a concrete example,  for any integer $n \ge 5$,  let $P(x_1,\dots,  x_n) = x_1^2 + \cdots + x_n^2 \in \mathbb{R}[x_1, \dots,  x_n]$.  Then $\overline{X} = \overline{Z_P(\mathbb{R})} = \{ (0,0, \dots, 0) \}$ is irreducible with codimension at least $5$.  Let $X \subseteq \mathbb{K}^n$ be the set of common zeros of $P_1,\dots, P_m \in \mathbb{K}[x_1,\dots,  x_n]$.  Since $\langle P_1,\dots,  P_m \rangle \subseteq \mathcal{I}(X)$,  it holds that
\[
\overline{X} \subseteq Z_{P_1,\dots,  P_m} \coloneqq \lbrace  
\mathsf{v} \in \overline{\mathbb{K}}^n: P_1(\mathsf{v}) = \cdots = P_m (\mathsf{v}) = 0
\rbrace.
\]
According to Krull’s principal ideal theorem,  each $\mathsf{v} \in Z_{P_1,\dots,  P_m}$ must be contained in some irreducible component of $ Z_{P_1,\dots,  P_m}$ of codimension at most $m$.  This does not ensure the existence of a high-dimensional irreducible subset of $\overline{X}(\mathbb{K})$ passing through $\mathsf{v} \in \overline{X}(\mathbb{K})$,  unless $\mathbb{K}$ is algebraically closed.  Therefore,  Theorem~\ref{thm:Krull} may be regarded as Krull’s principal ideal theorem for multilinear varieties.  As a direct consequence of Theorem~\ref{thm:Krull},  we derive Corollary~\ref{cor:fiber dim},  which is an analogue of the fiber dimension theorem \cite[Exercise~3.22]{hartshorne2013algebraic} for multilinear varieties.  

\subsection{Applications}\label{subsec:app}
Theorems~\ref{thm:codim} and \ref{thm:Krull} have several direct consequences,  which we briefly discuss in this subsection.  We denote by $\mathbb{K}[x_1,\dots,  x_n]_d$ the space of all homogeneous polynomials of degree $d$.  Let $f \in \Hom(\mathbb{K}^{n_1} \times \cdots \times \mathbb{K}^{n_d},  \mathbb{K})$ and $P_1,\dots,  P_m \in \mathbb{K}[x_1,\dots,  x_n]_d$.  Our applications are primarily concerned with the analytic rank $\AR_{\mathbb{K}}(f)$,  partition rank $\PR_{\mathbb{K}}(f)$,  geometric rank $\GR(f)$,  collective strength $\str_{\mathbb{K}}(P_1,\dots,   P_m)$ and collective Birch rank $\Brk(P_1,\dots,   P_m)$. For the precise definitions of these invariants,  we refer the reader to Definition~\ref{def-ar}.   

\subsubsection{Stability of ranks}
The first application concerns the stability of partition rank and strength under field extensions.  Given $f \in \Hom(\mathbb{K}^{n_1} \times \cdots \times \mathbb{K}^{n_d},  \mathbb{K})$ and a field extension $\mathbb{F}/\mathbb{K}$,  we denote by $f^{\mathbb{F}} \in \Hom( \mathbb{F}^{n_1} \times \cdots \times \mathbb{F}^{n_d},  \mathbb{F})$ the $\mathbb{F}$-multilinear function induced by $f$.  For convenience,  we simply write $R_{\mathbb{F}}(f) \coloneqq R_{\mathbb{F}}(f^{\mathbb{F}})$.  We say that $R$ has the \emph{stability} if $R_{\overline{\mathbb{K}}}(f) \asymp_d R_{\mathbb{K}}(f)$ for all $f \in \Hom(\mathbb{K}^{n_1} \times \cdots \times \mathbb{K}^{n_d},  \mathbb{K})$. Here ``$\asymp_{d}$" means the linear equivalence between two functions (cf. Definition~\ref{def:equiv}). We recall the following key conjecture on the stability of partition rank,  proposed by Adiprasito,  Kazhdan and Ziegler.
\begin{conjecture}[Stability of partition rank] \cite[Conjecture~1.7]{adiprasito2021schmidt}\label{conj:stabilitypr}
Let $\mathbb{K}$ be a field.  For any integer $d \ge 2$ and $f \in \Hom(\mathbb{K}^{n_1} \times \cdots \times \mathbb{K}^{n_d}, \mathbb{K})$, we have $\PR_{\mathbb{K}} (f) \asymp_d \PR_{\overline{\mathbb{K}}} (f)$.
\end{conjecture}
Over finite fields, Conjecture~\ref{conj:stabilitypr} is equivalent to Conjecture~\ref{conj:pvsa},  together with two other conjectures \cite[Theorem~5.6]{chen2024stability}.  By \cite[Claim~3.2]{LZ24}, Conjecture~\ref{conj:stabilitypr} implies the following:
\begin{conjecture}[Stability of collective strength]\cite[Conjecture~1.6]{LZ24}\label{conj:stabilitystr}
Let $\mathbb{K}$ be a field with $\ch(\mathbb{K}) = 0$ or $\ch(\mathbb{K}) > d$.  Then $\str_{\mathbb{K}} (P) \asymp_d \str_{ \overline{\mathbb{K}}} (P)$ for any $P \in \mathbb{K}[x_1,\dots,  x_n]_d$.
\end{conjecture}
Note that if $\ch(\mathbb{K}) = p \le d$ for some prime $p$,  then \cite[Example~1.2]{BDS25} shows that $\str_{\mathbb{K}}(P)$ can be arbitrarily larger than $\str_{\overline{\mathbb{K}}}(P)$,  implying the necessity of $\ch(\mathbb{K}) > d$ in Conjecture~\ref{conj:stabilitystr}.  On the other hand,  since $f \in \Hom(\mathbb{K}^{n_1} \times \cdots \times \mathbb{K}^{n_d},  \mathbb{K})$ can be viewed as a polynomial in $\mathbb{K}[x_{j,s}]_{j \in [d],\; s \in [n_j]}$, we have $\PR_{\mathbb{K}}(f) = O_d (\str_{\mathbb{K}}(f))$.  In particular,  Conjectures~\ref{conj:stabilitypr} and \ref{conj:stabilitystr} are equivalent over any field $\mathbb{K}$ with $\ch(\mathbb{K}) = 0$ or $\ch(\mathbb{K}) > d$.  However, Conjecture~\ref{conj:stabilitystr} is slightly weaker than Conjecture~\ref{conj:stabilitypr} because of the additional assumption that $\ch(\mathbb{K}) > d$. In Theorems~\ref{thm:AKZ} and \ref{thm:stability-collective-str},  we prove the stability of partition rank and collective strength over perfect infinite fields, respectively.  As before,  all constants are determined explicitly. For convenience,  we summarize the two results in the following theorem without presenting the explicit constants.
\begin{theorem}[Stability]\label{thm:stability}
Let $d \ge 2$ be an integer.  For any $f \in \Hom(\mathbb{K}^{n_1} \times \cdots \times \mathbb{K}^{n_d},  \mathbb{K})$ and $P,P_1,\dots,  P_m \in \mathbb{K}[x_1,\dots,  x_n]_d$,  we have 
\begin{enumerate}[label = (\alph*)]
\item If $\mathbb{K}$ is perfect and infinite,  then $\PR_{\mathbb{K}} (f) \asymp_d \PR_{\overline{\mathbb{K}}} (f)$.\label{thm:stability:item1}
\item If either $\ch(\mathbb{K}) = 0$ or $\mathbb{K}$ is perfect and infinite with $\ch(\mathbb{K}) > d$,  then $\str_{\mathbb{K}}(P_1,\dots,  P_m) \asymp_{d,m} \str_{\overline{\mathbb{K}}}(P_1,\dots,  P_m)$.  \label{thm:stability:item2}
\end{enumerate}
\end{theorem}
Two remarks are in order. Firstly,  Conjecture~\ref{conj:stabilitypr} is true for $d = 3$ over perfect fields \cite{chen2024stability,moshkovitz2024uniform,derksen2022g}. For $d \ge 4$,  while the conjecture has been extensively studied over finite fields \cite{cohen2023partition,moshkovitz2022quasi} and partially over infinite fields \cite{kazhdan2023schmidt,lampert2024relative,bik2025strength,baily2024strength}, it remains largely open for infinite fields. To the best of our knowledge, \ref{thm:stability:item1} is the first to achieve substantial progress in this direction, resolving Conjecture~\ref{conj:stabilitypr} over all perfect infinite fields.  Secondly,  \cite[Theorem~1.5.3]{bik2025strength} implies that for any field $\mathbb{K}$ with $\ch(\mathbb{K}) = 0$ or $\ch(\mathbb{K}) > d$,  the following bound holds for all $P_1,\dots,  P_m \in \mathbb{K}[x_1,\dots, x_n]_d$:
\begin{equation}\label{eq:collstrstability}
\str_{\mathbb{K}} (P_1,\dots, P_m) = \begin{cases}
m^3  O_d(r^{d-1}) \quad &\text{if~} \mathbb{K} \text{~is infinte}, \\
m^3  O_d(r^{d-1} \log (r + m)) \quad &\text{if~} \mathbb{K} \text{~is finite}.
\end{cases}
\end{equation}
where $r \coloneqq \str_{\overline{\mathbb{K}}} (P_1,\dots, P_m)$.  Polynomial bounds that are similar to that in \eqref{eq:collstrstability} are also established for semi-perfect fields \cite[Theorem~1.3]{BDS25} and admissible fields \cite[Theorem~1.5]{LZ24}. In contrast,  for infinite perfect fields, \ref{thm:stability:item2} improves these polynomial bounds to a linear bound,  which partially resolves Conjecture~\ref{conj:stabilitystr}. 

\subsubsection{Linear equivalence of ranks} 

For easy reference,  we state two important conjectures in the arithmetic and combinatorial study of multilinear functions and polynomials. 
\begin{conjecture}[Partition rank vs.  Analytic rank]\cite[Conjecture~1.10]{adiprasito2021schmidt}\label{conj:pvsa}
Let $\mathbb{K}$ be a finite field and let $d \ge 2$ be an integer.  Then $\PR_{\mathbb{K}}(f) \asymp_d \AR_{\mathbb{K}}(f)$,  for any $f\in \Hom(\mathbb{K}^{n_1} \times \cdots \times \mathbb{K}^{n_d},\mathbb{K})$.
\end{conjecture}
Here $\AR_{\mathbb{K}}(f)$ denotes the analytic rank \cite{gowers2011linear} of $f$ over the finite field $\mathbb{K}$.  It is proved in \cite{moshkovitz2022quasi}  that Conjecture~\ref{conj:pvsa} is true up to a log-factor. According to \cite{baily2024strength},  the validity of Conjecture~\ref{conj:pvsa} implies that of the following conjecture over finite fields.
\begin{conjecture}[Strength vs.  Birch rank] \cite[Conjecture~1.3]{baily2024strength}\label{conj:svsb}
Let $\mathbb{K}$ be a field and let $d \ge 2$ be an integer.  Suppose that either $\ch(\mathbb{K}) = 0$ or $\ch(\mathbb{K}) > d$.  Then for any $P \in \mathbb{K}[x_1,\dots, x_n]_d$,  we have 
$\str_{\mathbb{K}} (P) \asymp_d \Brk(P)$.
\end{conjecture}
In fact, the implication in \cite{baily2024strength} also shows that if $\PR_{\mathbb{K}}(f) \asymp_d \GR_{\mathbb{K}}(f)$ over any infinite field $\mathbb{K}$, then Conjecture~\ref{conj:svsb} holds for infinite fields.  Again, as we remarked in the discussion of the relation between Conjectures~\ref{conj:stabilitypr} and \ref{conj:stabilitystr}, Conjecture~\ref{conj:svsb}, though closely related to the linear equivalence between partition rank and geometric rank, is slightly weaker. We establish linear equivalences among $\PR_{\mathbb{K}}(f)$, $\AR_{\mathbb{K}}(f)$ and $\GR(f)$, as well as the linear equivalence between $\str_{\mathbb{K}}(P_1,\dots,  P_m)$ and $\Brk(P_1,\dots,  P_m)$, over perfect infinite fields.  
\begin{theorem}[Linear equivalence]\label{thm:linear}
Let $d \ge 2$ be an integer.  For any $f \in \Hom(\mathbb{K}^{n_1} \times \cdots \times \mathbb{K}^{n_d},  \mathbb{K})$ and $P,  P_1,\dots,  P_m \in \mathbb{K}[x_1,\dots,  x_n]_d$,  we have 
\begin{enumerate}[label = (\alph*)]
\item If $\mathbb{K}$ is infinite,  $\PR_{\mathbb{K}} (f) \asymp_d \AR_{\mathbb{K}} (f)   $.
 \label{thm:linear:item1}
\item If $\mathbb{K}$ is perfect and infinite,  then  $\PR_{\mathbb{K}}(f) \asymp_d \GR(f)$.
\label{thm:linear:item2}
\item If either $\ch(\mathbb{K}) = 0$ or $\mathbb{K}$ is perfect and infinite with $\ch(\mathbb{K}) > d$,  then $\str_{\mathbb{K}} (P_1,\dots,  P_m) \asymp_{d,m} \Brk(P_1,\dots,  P_m)$.\label{thm:linear:item3}
\end{enumerate}
\end{theorem}
The linear equivalences in \ref{thm:linear:item1}--\ref{thm:linear:item3} are sequentially proved in Theorems~\ref{thm:pvsa}, \ref{thm:pvsg} and \ref{thm:collective str-birc}. Moreover,  all constants depending on $d$ and $m$ in \ref{thm:linear:item1}-\ref{thm:linear:item3} are explicitly determined.  

We conclude this subsection by mentioning several existing results related to \ref{thm:linear:item1}-\ref{thm:linear:item3}.  Suppose that $\mathbb{K}$ is an infinite field.  According to \cite[Theorem~A.1]{LZ24} and the proof of Theorem~2.3 in \cite{baily2024strength},  we have $\PR_{\mathbb{K}} (f) = O_d(\AR_{\mathbb{K}} (f))$ and $\PR_{\mathbb{K}} (f) \asymp_{d,\mathbb{K}} \GR(f)$,  which are respectively improved by \ref{thm:linear:item1} and \ref{thm:linear:item2}.  In particular,  \ref{thm:linear:item1} serves as an analogue of Conjecture~\ref{conj:pvsa} over infinite fields,  and \ref{thm:linear:item2} completely removes the dependence on $\mathbb{K}$ in the equivalence $\PR_{\mathbb{K}} (f) \asymp_{d,\mathbb{K}} \GR(f)$. By \cite[Theorem~1.4]{baily2024strength}, it holds that $\str_{\mathbb{K}} (P) \asymp_{d,\mathbb{K}} \Brk (P)$.  The $m = 1$ case of the linear equivalence in \ref{thm:linear:item3} removes the field dependence, and it resolves Conjecture~\ref{conj:svsb} for perfect infinite fields.

\subsubsection{Abundance of rational points} Let $P_1,\dots,  P_m \in \mathbb{K}[x_1,\dots,  x_n]$ be homogeneous polynomials of degrees $e_1,\dots,  e_m$,  respectively.  As generalizations of the celebrated theorems of Brauer \cite{Brauer45} and Birch \cite{Birch57}, it has recently been proved in \cite[Theorem~1.3]{AJA24} and \cite[Theorem~1.3]{AA24} that
\begin{equation}\label{eq:BrauerBirch}
\dim Z_{P_1,\dots, P_m} - \dim \overline{Z_{P_1,\dots, P_m}(\mathbb{K})} \le C(e_1,\dots, e_m)
\end{equation}
for some function $C$,  provided that $\mathbb{K}$ is either an infinite Brauer field or a Birch field of characteristic zero. We refer the reader to \cite{AA24} for the precise definitions of Brauer and Birch fields.  It is noteworthy that neither \cite[Theorem~1.3]{AJA24} nor \cite[Theorem~1.3]{AA24} explicitly determines the function $C$,  and that \eqref{eq:BrauerBirch} provides an upper bound for the relative codimension of $\overline{Z_{P_1,\dots, P_m}(\mathbb{K})} $,  rather than for $\codim \overline{Z_{P_1,\dots, P_m}(\mathbb{K})}$.  As a consequence of Theorem~\ref{thm:Krull},  we immediately have the following strengthened and effective version of \eqref{eq:BrauerBirch} for multilinear varieties.
\begin{corollary}[Abundance of rational points]
Let $\mathbb{K}$ be an infinite field. Suppose that $f_1,\dots,  f_m: \mathbb{K}^{n_1} \times \cdots \times \mathbb{K}^{n_d} \to \mathbb{K}$ are functions such that $\mdeg f_1,  \dots,  \mdeg f_m \le (1,\dots, 1)$.  Then for each irreducible component $W$ of $\overline{Z_{f_1,\dots, f_m}(\mathbb{K})}$,  we have $\codim W \le 2^d m$. 
\end{corollary}

\subsection{Organization} We establish some basic results in Section~\ref {sec:prelim}, which are essential ingredients in our proofs of Theorems~\ref {thm:codim} and \ref {thm:Krull} in Section~\ref {sec:geometry}.  Proofs of the aforementioned applications are presented in Section~\ref{sec:app},  along with further qualitative results concerning ranks of tensors and polynomials. 

\section{Preliminary lemmas}\label{sec:prelim}
This section is devoted to gathering preliminary facts for subsequent use.  While several results are well-known or straightforward to experts, we supply full proofs for the sake of completeness.
\subsection{Properties of the Zariski topology}
We first record two basic facts of irreducible subsets in the following lemma.   
\begin{lemma}\label{additional lemma 1}
Let $\mathbb{K}$ be an infinite field and let $S$ be a subset of $\mathbb{K}^{n}$.  Suppose that $\overline{S}$ is irreducible.  
\begin{enumerate}[label = (\alph*)]
\item For any non-empty open subset $U\subseteq \overline{S}$,  we have $\overline{U\cap S}=\overline{S}$.
\item If $\varphi: \mathbb{K}^n \mydashrightarrow \mathbb{K}^m$ is a rational map such that $\varphi$ is defined on $S$,  then $\overline{ \varphi (S)}$ is irreducible.
\end{enumerate}
\end{lemma}
\begin{proof}
We denote $Z \coloneqq \overline{S}$.  
\begin{enumerate}[label = (\alph*)]
\item Suppose on the contrary that $W \coloneqq \overline{U\cap S} \subsetneq Z$.  Since $U \cap S \subseteq W$, $(Z \setminus W) \cap U \cap S = \varnothing$.  However,  this is not possible since $(Z \setminus W) \cap U$ is an open dense subset of $Z = \overline{S}$.
\item Let $E \subseteq \mathbb{K}^n$ be the indeterminacy of $\varphi$.  Then $\varphi$ is a regular map on the open subset $U \coloneqq \mathbb{K}^n \setminus E$.  Since $\overline{S}$ is irreducible,  so is $S$.  The continuity of $\varphi$ on $S \subseteq U$ implies the irreducibility of $\varphi(S)$,  from which we obtain the irreducibility of $\overline{\varphi(S)}$.  
\end{enumerate}
\end{proof}

Next,  we establish Proposition~\ref{prop:dimension} which compares the dimension of two varieties via rational maps.  To achieve this,  we need the lemma that follows.  
\begin{lemma}\label{lem-1}
Let $\mathbb{K}$ be an algebraically closed field and let $S \subseteq\mathbb{K}^n$ be a subset. If $Z \coloneqq \overline{S} =\bigcup_{i \in [s]} Z_i$ is the irreducible decomposition of $Z$, then $\overline{S \cap Z_k}= Z_k$ for each $k \in [s]$.
\end{lemma}
\begin{proof}
For each $k \in [s]$,  it is clear that $\overline{S \cap Z_k}\subseteq Z_k$.  Thus,  it suffices to prove the reversed inclusion.  We observe that
\[
Z_k \subseteq   Z =  \overline{\bigcup_{j \in [s]} (S \cap Z_j)}=\bigcup_{j \in [s]} \overline{S \cap Z_j}.
\]
The irreducibility of $Z_k$ implies that $Z_k\subseteq\overline{S \cap Z_{j_k}} \subseteq Z_{j_k}$ for some $j_k \in [s]$.  Thus,  we obtain $j_k = k$ and $Z_k = \overline{S \cap Z_k}$.
\end{proof}
\begin{proposition}[Dimension comparison]\label{prop:dimension}
Let $\mathbb{K}$ be an algebraically closed field.  Suppose that $S \subseteq \mathbb{K}^n$,  $T \subseteq \mathbb{K}^m$ are two subsets.  If $\varphi:\mathbb{K}^n \mydashrightarrow \mathbb{K}^m$ is a rational map such that $\varphi$ is defined on $S$ and $\varphi(S)= T$,  then $\dim(\overline{S})\ge\dim(\overline{T})$.
\end{proposition}
\begin{proof}
Denote $Z \coloneqq \overline{S}$ and $W \coloneqq \overline{T}$.  Assume that $Z = \bigcup_{j \in [s]} Z_j$ is the irreducible decomposition of $Z$ and that $E \subseteq \mathbb{K}^{n}$ is the indeterminacy locus of $\varphi$.  By definition,  $E$ is a Zariski closed subset of $\mathbb{K}^n$,  and $S \subseteq Z \setminus E$.

For each $j \in [s]$, Lemma~\ref{lem-1} implies that $S \cap Z_j \neq \varnothing$ .  Therefore,  we may conclude that $Z_j \setminus E$ is an open dense subset of $Z_i$.  Moreover,  $\varphi: Z_j \setminus E \to \overline{\phi(Z_j \setminus E)}$ is regular and dominant.  Consequently,  we obtain $\dim Z_j  = \dim (Z_j \setminus E) \ge \dim \overline{\varphi(Z_j \setminus E)}$. We notice that 
\[
W =\overline{T}=\overline{\varphi(S)}\subseteq\overline{\varphi(Z \setminus E)}=\overline{\varphi\Big(
\bigcup_{j \in [s]}(Z_j\setminus E)\Big)} = \bigcup_{j \in [s]}\overline{\varphi\left(Z_j \setminus E\right)},
\]
which implies
    \[
    \dim W \le \max_{j \in [s]} \big\{ \dim \overline{\varphi(Z_j\setminus E)} \big\} \le \max_{j \in [s]} \big\{ \dim Z_j \big\} =\dim(Z).  \qedhere
    \]
\end{proof}
As a direct consequence of Proposition~\ref{prop:dimension},  we have the following:
\begin{corollary}\label{cor-1}
Let $\mathbb{K}$ be an algebraically closed field.  Suppose that $S \subseteq\mathbb{K}^n$ and $T\subseteq\mathbb{K}^m$ are subsets,  and there exist rational maps $\varphi: \mathbb{K}^n \mydashrightarrow \mathbb{K}^m$ and $\psi:\mathbb{K}^m \mydashrightarrow  \mathbb{K}^n$ such that 
\begin{enumerate}[label = (\alph*)]
\item $\varphi$ is defined on $S$ and $\psi$ is defined on $T$.
\item $\varphi(S)=T$ and $\psi(T)=S$.
\end{enumerate}
Then it holds that $\dim \overline{S}=\dim \overline{T}$.
\end{corollary}

\subsection{Geometry of fixed rank matrices}
This subsection is concerned with the geometry of matrices of fixed rank.  Let $\mathbb{K}$ be a field and let $n,m$ be two positive integers.  Suppose that $I$ and $J$ are subsets of $[n]$ and $[m]$ such that $|I| = |J| = r \le \min \{n,m\}$.  Denote 
\begin{align}
        Z(I,J) &\coloneqq \{(\mathsf{M},  \mathsf{v})\in \mathbb{K}^{n\times m} \times \mathbb{K}^{m}: \mathsf{M} \mathsf{v}=0,\; \rank \mathsf{M} = \rank \mathsf{M}_{I,J} =  r\},\label{eq:ZIJ}\\
        W(I,J) &\coloneqq \{\mathsf{M} \in \mathbb{K}^{n\times m}:  \rank \mathsf{M} = \rank \mathsf{M}_{I,J} =  r \}.  \label{eq:WIJ}
    \end{align} 

\begin{proposition}[Fixed rank matrices]\label{prop:fixed rank}
There are rational maps $\varphi: \mathbb{K}^{n\times m} \times \mathbb{K}^{m} \mydashrightarrow \mathbb{K}^{n\times m} \times \mathbb{K}^{m - r}$ and $\psi: \mathbb{K}^{n\times m} \times \mathbb{K}^{m - r} \mydashrightarrow \mathbb{K}^{n\times m} \times \mathbb{K}^{m}$ satisfying the following properties: 
\begin{enumerate}[label = (\alph*)]
\item $\varphi$ is defined on $Z(I,J)$ and $\psi$ is defined on $W(I,J) \times \mathbb{K}^{m -r}$. 
\item $\varphi(Z(I,J)) = W(I,J) \times \mathbb{K}^{m - r}$ and $\psi(W(I,J) \times \mathbb{K}^{m-r}) = Z(I,J)$.
\item $\psi \circ \varphi|_{Z(I,J)} = \Id_{Z(I,J)}$ and $\varphi \circ \psi|_{W(I,J) \times \mathbb{K}^{m - r}} = Z(I,J)$.
\end{enumerate}
\end{proposition}
\begin{proof}
Without loss of generality,  we may assume that $I = J =[r]$.  We partition each $\mathsf{M} \in \mathbb{K}^{n\times m}$ and $v \in \mathbb{K}^m$ as
\begin{equation}\label{prop:fixed rank:eq1}
    \mathsf{M} =\begin{bmatrix}
        \mathsf{M}_1 & \mathsf{M}_2\\
        \mathsf{M}_3& \mathsf{M}_4
    \end{bmatrix},  \quad \mathsf{v}=\begin{bmatrix}
        v_1\\
        v_2
    \end{bmatrix}, 
 \end{equation}
where $\mathsf{M}_1 \in \mathbb{K}^{r \times r}$,  $\mathsf{v}_1 \in \mathbb{K}^r$,  $\mathsf{v}_2 \in \mathbb{K}^{m- r}$,  and $\mathsf{M}_2$,  $\mathsf{M}_3$,  $\mathsf{M}_4$ are matrices of appropriate sizes.  We define 
\begin{align*}
&\varphi: \mathbb{K}^{n\times m} \times \mathbb{K}^{m} \myto \mathbb{K}^{n\times m} \times \mathbb{K}^{m - r},\quad \varphi(\mathsf{M},\mathsf{v}) = (\mathsf{M},  \mathsf{v}_2),  \\
&\psi: \mathbb{K}^{n\times m} \times \mathbb{K}^{m - r} \mydashrightarrow \mathbb{K}^{n\times m} \times \mathbb{K}^{m},\quad \psi(\mathsf{M},  \mathsf{w}) = \Big( \mathsf{M},\begin{bmatrix}
        -\det(\mathsf{M}_1)^{-1} \adj(\mathsf{M}_{1}) \mathsf{M}_2 \mathsf{w}\\
        \mathsf{w}
    \end{bmatrix} \Big).
\end{align*}
It is clear that $\psi$ is defined on $W([r],[r])\times \mathbb{K}^{m - r}$,  $\varphi \circ \psi|_{W([r],[r])\times \mathbb{K}^{m - r}} = \Id_{W([r],[r])\times \mathbb{K}^{m - r}}$ and $\psi(W([r],[r]) \times \mathbb{K}^{m-r}) \subseteq Z([r],[r])$.  For each $(\mathsf{M},\mathsf{v}) \in Z([r],[r])$,  we have
\[
\mathsf{v}_1 = -\mathsf{M}_1^{-1}\mathsf{M}_2 \mathsf{v}_2,\quad \rank \mathsf{M} = \rank \mathsf{M}_1 = r. 
\]
Thus,  $\psi(\mathsf{M}, \mathsf{v}_2) = (\mathsf{M},  \mathsf{v})$ and $\psi(W([r],[r]) \times \mathbb{K}^{m-r}) = Z([r],[r])$.  This implies that $\varphi (Z([r],[r])) = W([r],[r]) \times \mathbb{K}^{m-r}$ and $\psi \circ \varphi|_{Z([r],[r])} = \Id_{Z([r],[r])}$.
\end{proof}

\begin{remark}
Investigating the geometry of low rank matrices via rational maps is an important technique in the study of multilinear maps.  Results analogous to Proposition~\ref{prop:fixed rank} can be found in \cite[Lemma~4.1]{cohen2023partition} and \cite[Lemma~3.1]{moshkovitz2022quasi}. The set $W(I,J)$ in \eqref{eq:WIJ} can be described explicitly by the Schur complement.  For simplicity,  we consider $I = J  = [r]$ and partition $\mathsf{M} \in W(I,J)$ as in \eqref{prop:fixed rank:eq1}.  Then we have 
\[
    \begin{bmatrix}
        \Id_r &0\\
        -\mathsf{M}_3 \mathsf{M}_1^{-1} & \Id_{n - r}
    \end{bmatrix}     \begin{bmatrix}
        \mathsf{M}_1 & \mathsf{M}_2 \\
        \mathsf{M}_3 & \mathsf{M}_4
    \end{bmatrix} =  \begin{bmatrix}
        M_1 & M_2 \\
        0 & M_4 - M_3 M_1^{-1} M_2
    \end{bmatrix}.
\]
Since $\rank \mathsf{M}_1 = \rank \mathsf{M}$,  $\mathsf{M}_4 = \mathsf{M}_3 \mathsf{M}_1^{-1} \mathsf{M}_2$ and $\mathsf{M} = \begin{bsmallmatrix} 
\mathsf{M}_1 & \mathsf{M}_2 \\
\mathsf{M}_3 & \mathsf{M}_3 \mathsf{M}_1^{-1} \mathsf{M}_2
\end{bsmallmatrix}$.  As a consequence of Proposition~\ref{prop:fixed rank},  $Z(I,J)$ is a trivial vector bundle over $W(I,J)$ via the projection map $\pi: Z(I,J)\to W(I,J)$ defined by $\pi(\mathsf{M},v) = \mathsf{M}$.
\end{remark}

\section{Geometry of multilinear varieties}\label{sec:geometry}
Let $\mathbb{K}$ be a field.  Suppose that $f_1,\dots,  f_m: \mathbb{K}^{n_1} \times \cdots \times \mathbb{K}^{n_d} \to \mathbb{K}$ are maps such that $\mdeg f_j \le (1,\dots,  1)$,  $j \in [m]$.  We recall the following:
\begin{definition}[Multilinear variety]\label{def:multivar}\cite{GM21} 
The \emph{multilinear variety} over $\mathbb{K}$ defined by $f_1,\dots,  f_m$ is: 
\[
Z_{f_1,\dots,  f_m} (\mathbb{K}) \coloneqq \lbrace
(\mathsf{v}_1,\dots,  \mathsf{v}_d) \in \mathbb{K}^{n_1} \times \cdots \times \mathbb{K}^{n_d}:  f_j(\mathsf{v}_1,\dots,  \mathsf{v}_d) = 0,\; j \in [m]
\rbrace.
\]
Given $F = (f_1,\dots,  f_m) \in \Hom(\mathbb{K}^{n_1} \times \cdots \times \mathbb{K}^{n_d},  \mathbb{K}^{m})$,  we denote  $Z_F(\mathbb{K}) \coloneqq Z_{f_1,\dots,  f_m}(\mathbb{K})$.
\end{definition}
The purpose of this section is to establish the codimension formula (cf.  Theorem~\ref{thm:codim}) and the existence of high-dimensional subvariety (cf.  Theorem~\ref{thm:Krull}),  for multilinear varieties.

\subsection{Proof of Theorem~\ref{thm:codim}}\label{subsec:codim}
The main purpose of this subsection is to establish the formula \eqref{thm:codim:eq} for $\codim \overline{Z_F(\mathbb{K})}$. To begin with,  we present the following example to illustrate the difference between $\codim \overline{Z_F(\mathbb{K})}$ and $\height \mathcal{I}(Z_F(\mathbb{K}))$.
\begin{example}[Central division algebras]\label{example1}
Let $\mathcal{D}$ be a finite dimensional division algebra over its center $\mathbb{K}$.  By \cite[Corollary~2.1.7]{GS17},  $\dim_{\mathbb{K}} \mathcal{D} = r^2$ and $ \mathcal{D} \otimes_{\mathbb{K}} \overline{\mathbb{K}} \simeq \overline{\mathbb{K}}^{r \times r}$ as $\mathbb{K}$-algebras,  for some positive integer $r$.  Let $F: \mathcal{D} \times \mathcal{D} \to \mathcal{D}$ be the $\mathbb{K}$-bilinear map defined by $F(\alpha,  \beta) \coloneqq \alpha \beta$.  Since $\mathcal{D}$ is a division ring,  the multilinear variety defined by $F$ is 
\[
Z_F(\mathbb{K}) = \big( \mathcal{D} \times \{0\} \big) \cup \big( \{0\}  \times  \mathcal{D} \big).
\] 
Thus,  $\codim \overline{Z_F(\mathbb{K})} = r^2$.  On the other hand,  since $\mathcal{D} \otimes_{\mathbb{K}} \overline{\mathbb{K}} \simeq \overline{\mathbb{K}}^{r \times r}$,  \cite[Theorem~6.1]{kopparty2020geometric} implies $\GR(F) = \lceil 3r^2/4 \rceil$.  In particular,  $\codim \overline{Z_F(\mathbb{K})}  > \GR(F)$,  whenever $r \ge 2$.  Non-split quaternion algebras \cite[Definition~2.2.1]{Voight21} are typical such examples.  We refer the reader to \cite{Amitsur72,Tignol86,Amitsur88} for explicit examples of arbitrarily large dimensions.
\end{example}

For each $r \in \mathbb{N}$ and $F \in \Hom(\mathbb{K}^{n_1} \times \cdots \times \mathbb{K}^{n_{k}},  \mathbb{K}^m)$,  we define a subset
\begin{equation}\label{eq:ZrF}
Z_{r,F}(\mathbb{K}) \coloneqq \{ (\mathsf{v}_1,\dots,  \mathsf{v}_k) \in Z_F(\mathbb{K}): \rank F(\mathsf{v}_1,\dots,  \mathsf{v}_{k-1},\cdot) = r \},
\end{equation}
and proceed to the proof of Theorem~\ref{thm:codim}. 
\begin{proof}[Proof of Theorem~\ref{thm:codim}]
Suppose that $I$ and $J$ are subsets of $[n_d]$ and $[m]$ such that $|I| = |J| = r$.  Let $Z(I,J) \subseteq \mathbb{K}^{n_d \times m} \times \mathbb{K}^m$ and $W(I,J) \subseteq \mathbb{K}^{n_d \times m}$ be defined as in \eqref{eq:ZIJ} and \eqref{eq:WIJ},  respectively.  Suppose $\varphi: \Hom(\mathbb{K}^{n_d}, \mathbb{K}^m) \times \mathbb{K}^{n_d} \mydashrightarrow \Hom(\mathbb{K}^{n_d}, \mathbb{K}^m) \times \mathbb{K}^{n_d - r}$ and $\psi: \Hom(\mathbb{K}^{n_d}, \mathbb{K}^m) \times \mathbb{K}^{n_d - r} \mydashrightarrow \Hom(\mathbb{K}^{n_d}, \mathbb{K}^m) \times \mathbb{K}^{n_d}$ are rational maps for $Z(I,J)$ and $W(I,J)$ in Proposition~\ref{prop:fixed rank}.

We consider the maps
\begin{align*}
&\xi_F:  \mathbb{K}^{n_1} \times \cdots \times  \mathbb{K}^{n_d} \myto \Hom(\mathbb{K}^{n_d},  \mathbb{K}^m) \times \mathbb{K}^{n_d},\quad (v_1,\dots,  v_d) \mapsto (F(v_1,\dots,  v_{d-1}, \cdot),v_d),  \\
&\eta_F: \mathbb{K}^{n_1} \times \cdots \times \mathbb{K}^{n_{d -1}} \myto \Hom(\mathbb{K}^{n_d},  \mathbb{K}^m),\quad (v_1,\dots,  v_{d - 1}) \mapsto F(v_1,\dots,  v_{d-1}, \cdot). 
\end{align*}
For each $(\mathsf{v}_1,\dots,  \mathsf{v}_{d-1},  \mathsf{v}_d) \in \mathbb{K}^{n_1} \times \cdots \times \mathbb{K}^{n_d}$,  we identify $F(\mathsf{v}_1,\dots, \mathsf{v}_{d-1},\cdot) \in \Hom(\mathbb{K}^{n_d},  \mathbb{K}^m) $ with a matrix $\mathsf{M} \in \mathbb{K}^{n_k \times m}$,  and  partition $\mathsf{v}_d$ and $\mathsf{M}$ as 
\begin{align*}
\mathsf{v}_d &= \begin{bmatrix}
\mathsf{v}_{d,1} \\
\mathsf{v}_{d,2}
\end{bmatrix},  \mathsf{v}_{d,1} \in \mathbb{K}^{r},\; \mathsf{v}_{d,2} \in \mathbb{K}^{n_d - r},  \\
\mathsf{M} &= \begin{bmatrix}
\mathsf{M}_1 & \mathsf{M}_2 \\
\mathsf{M}_3 & \mathsf{M}_4
\end{bmatrix},\quad \mathsf{M}_1 \in \mathbb{K}^{r \times r}, \; \mathsf{M}_2 \in \mathbb{K}^{r \times (m-r)},\; \mathsf{M}_3 \in \mathbb{K}^{(n_d-r) \times r},\;  \mathsf{M}_4 \in \mathbb{K}^{(n_d -r) \times (m-r)}.
\end{align*}

If we denote $X_F(I,J) \coloneqq \xi_F^{-1}(Z(I,J))$ and $Y_F(I,J) \coloneqq \eta_F^{-1} (W(I,J))$,  then it is clear that 
\[
Z_{r,F}(\mathbb{K}) = \bigcup_{r,I,J} X_{r,F} (I,J),\quad W_{r,F}(\mathbb{K}) = \bigcup_{r,I,J} Y_{r,F} (I,J),
\]
where $r$ ranges over integers between $0$ and $\min\{n_d,m\}$,  and $I$ (resp.  $J$) runs through subsets of $[n_k]$ (resp.  $[m]$) of cardinality $r$.  By construction,  we also have $Z_F(\mathbb{K}) = \bigcup_{r} Z_{r,F}(\mathbb{K})$.  

We further define two maps $\Phi: \mathbb{K}^{n_1} \times \cdots \times  \mathbb{K}^{n_d} \to \mathbb{K}^{n_1} \times \cdots \times \mathbb{K}^{n_{d-1}}\times \mathbb{K}^{n_d - r} $ and $\Psi: \mathbb{K}^{n_1} \times \cdots \times \mathbb{K}^{n_{d-1}}\times \mathbb{K}^{n_d - r} \mydashrightarrow  \mathbb{K}^{n_1} \times \cdots \times  \mathbb{K}^{n_d}$ by 
\begin{align*}
\Phi(\mathsf{v}_1,\dots,  \mathsf{v}_d) &\coloneqq (\mathsf{v}_1,\dots,  \mathsf{v}_{d-1},  \mathsf{v}_{d,2}), \\ 
\Psi(\mathsf{v}_1,\dots, \mathsf{v}_{d-1},  \mathsf{w}) &\coloneqq \Big( \mathsf{v}_1,\dots,  \mathsf{v}_{d -1},  \begin{bmatrix}
-\det(\mathsf{M}_1)^{-1} \adj(\mathsf{M}_1) \mathsf{M}_2 \mathsf{w} \\
\mathsf{w}
\end{bmatrix} \Big).
\end{align*}
Then we have the following commutative diagram:
\[\begin{tikzcd}
	{ \mathbb{K}^{n_1} \times \cdots \times  \mathbb{K}^{n_d}} && {\Hom(\mathbb{K}^{n_d}, \mathbb{K}^m) \times \mathbb{K}^{n_d}} \\
	{\mathbb{K}^{n_1} \times \cdots \times \mathbb{K}^{n_{d -1}} \times \mathbb{K}^{n_d -r}} && {\Hom(\mathbb{K}^{n_d}, \mathbb{K}^m) \times \mathbb{K}^{n_d - r}}
	\arrow["{\xi_F}", from=1-1, to=1-3]
	\arrow["\Phi", dashed, from=1-1, to=2-1]
	\arrow["\varphi", dashed, from=1-3, to=2-3]
	\arrow["\Psi", shift left=3, dashed, from=2-1, to=1-1]
	\arrow["{\eta_F \times \Id_{\mathbb{K}^{n_k - r}}}"', from=2-1, to=2-3]
	\arrow["\psi", shift left=3, dashed, from=2-3, to=1-3]
\end{tikzcd}.\]
Using properties of $\varphi$ and $\psi$,  it is straightforward to verify that 
\[
\Phi(X_{r,F}(I,J)) = Y_{r,F}(I,J) \times \mathbb{K}^{n_d - r},\quad \Psi( Y_{r,F}(I,J) \times \mathbb{K}^{n_d - r} ) =X_{r,F}(I,J).
\]
According to Corollary~\ref{cor-1},  we obtain
\[
\dim \overline{X_{r,F}(I,J)} = \dim \big( \overline{Y_{r,F}(I,J)} \times \overline{\mathbb{K}}^{n_d - r} \big) = \dim \overline{Y_{r,F}(I,J)} + n_d - r.
\]
Thus,  we obtain 
\begin{align*}
\dim \overline{Z_F(\mathbb{K})} = \max_{r,I,J} \Big\lbrace \dim  \overline{X_{r,F}(I,J)}  \Big\rbrace &= \max_{r,I,J} \Big\lbrace \dim  \overline{Y_{r,F}(I,J)} + n_d - r \Big\rbrace \\
&=n_d  +  \max_{r} \Big\lbrace  \dim  \overline{W_{r,F}(\mathbb{K})}  - r \Big\rbrace,
\end{align*}
which implies that $\codim \overline{Z_F(\mathbb{K})} = \max_{r} \big\lbrace \codim \overline{W_{r,F}(\mathbb{K})}  + r  \big\rbrace$.
\end{proof}

\subsection{Proof of Theorem~\ref{thm:Krull}}\label{subsec:subvar}
Next,  we establish Theorem~\ref{thm:Krull} concerning the existence of high-dimensional irreducible subvarieties of multilinear varieties.      
\begin{proof}[Proof of Theorem~\ref{thm:Krull}]
Denote $N \coloneqq 2^d m$.  Given $j \in [m]$ and $I \subseteq [d]$,  we define $f_{j,I} \coloneqq f_j(y_1,\dots,  y_d) - c_j$ where  
\[
y_i = \begin{cases}
\mathsf{v}_i + x_i \quad &\text{if~} i \not\in I,  \\
\mathsf{v}_i \quad &\text{if~} i \in I.
\end{cases}
\]
Then we have $\mdeg f_{j,I} \le (1,\dots,  1)$ and 
\[
f_{j,\varnothing}(x_1,\dots, x_d) = f_j(\mathsf{v}_1 + x_1,\dots,  \mathsf{v}_d + x_d) - c_j,\quad f_{j, I}(0,\dots,  0) = f_j(\mathsf{v}_1,\dots,  \mathsf{v}_d) - c_j = 0.
\]
We consider 
\[
Z_0 = \lbrace \mathsf{w} \in \mathbb{K}^{n_1} \times \cdots \times \mathbb{K}^{n_d}: f_{j,I}(\mathsf{w}) = 0,\; j\in [m],\; I\subseteq [d]  \rbrace.
\]
It is clear that $Z_0 \subseteq Z$.  It is sufficient to prove the existence of $W \subseteq Z_0$ satisfying \ref{thm:Krull:cond1}--\ref{thm:Krull:cond3} with $\mathsf{v} = 0$ and $c_1 = \cdots = c_m = 0$.

For each $s \in [d]$,  we denote
\[
\mathsf{e}_s \coloneqq (\underbrace{1,\dots, 1}_{s \text{~copies}},  \underbrace{0,\dots, 0}_{d-s~\text{~copies}}).
\]
We prove the following claim  by induction  on $s$: for any $s \in [d]$,  there exists $W_s \subseteq \mathbb{K}^{n_1} \times \cdots \times \mathbb{K}^{n_s}$ with the following properties: 
\begin{enumerate}[label = (\roman*)]
\item If $\mdeg f_{j,I} \le \mathsf{e}_s$ for some $j \in [m]$ and $I\subseteq [d]$,  then $f_{j,I}$ vanishes on $W_s$.  \label{thm:Krull:eqi}
\item $\lambda W_s = W_s$ for any $\lambda \in \mathbb{K}^{\times}$.   \label{thm:Krull:eqii}
\item $\overline{W}_s$ is irreducible.   \label{thm:Krull:eqiii}
\item $n_1 + \cdots + n_s - \dim \overline{W_s} \le \delta_s$,  where $\delta_s$ is the number of pairs consisting of $j \in [m]$ and $I \subseteq [d]$ such that $\mdeg f_{j,I} \le \mathsf{e}_s$.   \label{thm:Krull:eqiv}
\end{enumerate}
Taking $W \coloneqq W_d$,  \ref{thm:Krull:eqi} ensures that $W \subseteq Z_0$,  \ref{thm:Krull:eqii} implies that $0 \in \overline{W}$,  and $\overline{W}$ is an irreducible variety of codimension at most $\delta_{d} = 2^d m$.

If $\mdeg f \le \mathsf{e}_1$,  then $f$ is either a constant or a linear function on $\mathbb{K}^{n_1}$.  Therefore,  the claim is true for $s = 1$ if we take $W_1$ to be the linear subspace
\[
W_1 = \{ \mathsf{v} \in \mathbb{K}^{n_1} \times \cdots \times \mathbb{K}^{n_s}: f_{j,I}(\mathsf{v}) = 0,\; \mdeg f_{j,I} \le \mathsf{e}_1 \}.
\]
Next,  we suppose that the claim is true for $s = s_0 < d$ and we prove for $s = s_0 + 1$. For simplicity,  we rename $f_{j,I}$'s as $g_1,\dots,  g_{\delta_k}$ so that 
\[
\mdeg g_j  \begin{cases}
\le \mathsf{e}_{s_0} \quad &\text{if~} j \in [\delta_{s_0}], \\
\le \mathsf{e}_{s_0+1} \quad &\text{if~} j \in [\delta_{s_0+1}] \setminus [\delta_{s_0}], \\
\ge \mathsf{e}_{s_0 + 2}  \quad &\text{otherwise}.
\end{cases}
\]
We consider the map 
\[
M: \mathbb{K}^{n_1} \times \cdots \times \mathbb{K}^{n_{s_0}} \to \Hom(\mathbb{K}^{n_{s_0 + 1}},  \mathbb{K}^{\delta_{s_0+1} -  \delta_{s_0}}),  \quad 
M(x_1,\dots,  x_{s_0}) = \begin{bmatrix}
g_{\delta_{s_0} + 1} (x_1,\dots,  x_{s_0},\cdot) \\
\vdots \\
g_{\delta_{s_0+1}} (x_1,\dots,  x_{s_0},\cdot)
\end{bmatrix}.
\]
By assumption,  $M$ is a nonzero map on $\overline{W}_{s_0}$.  The induction hypothesis implies that $\overline{W}_{s_0}$ is irreducible.  Thus, there is an $r \times r$ minor $f$ of $M$ such that $\rank M(\mathsf{w}_1,\dots,  \mathsf{w}_{s_0}) = r$ if $(\mathsf{w}_1,\dots, \mathsf{w}_{s_0}) \in V_{s_0} \coloneqq \overline{W}_{s_0} \cap D(f) \ne \varnothing$.  Here $D(f)$ denotes the open subset of $\mathbb{K}^{n_1} \times \cdots \times \mathbb{K}^{n_{s_0}}$ on which $f$ is non-vanishing.  We denote
\begin{align*}
X_{s_0} &\coloneqq (V_{s_0} \cap W_{s_0}) \times \mathbb{K}^{n_{s_0 + 1} - r}, \\
Y_{s_0} &\coloneqq \{ (\mathsf{w}_1,\dots,  \mathsf{w}_{s_0 + 1}) \in \mathbb{K}^{n_1} \times \cdots \times \mathbb{K}^{n_{s_0 + 1}}: (\mathsf{w}_1,\dots,  \mathsf{w}_{s_0}) \in  V_{s_0} \cap W_{s_0},\; M(\mathsf{w}_1,\dots,  \mathsf{w}_{s_0})(\mathsf{w}_{s_0 + 1}) = 0 \}.
\end{align*}
The same argument as that in the proof of Theorem~\ref{thm:codim} implies that there are rational maps $\Phi:\mathbb{K}^{n_1} \times \cdots \times \mathbb{K}^{n_{s_0+1}}\mydashrightarrow \mathbb{K}^{n_1} \times \cdots \times \mathbb{K}^{n_{s_0}} \times \mathbb{K}^{n_{s_0 + 1} - r}$ and $\Phi:\mathbb{K}^{n_1} \times \cdots \times \mathbb{K}^{n_{s_0}} \times \mathbb{K}^{n_{s_0 + 1} - r} \mydashrightarrow \mathbb{K}^{n_1} \times \cdots \times \mathbb{K}^{n_{s_0+1}}$ such that $\Phi$ (resp.  $\Psi$) is defined on $Y_{s_0}$ (resp.  $X_{s_0}$),  and $\Phi(Y_{s_0}) = X_{s_0}$ (resp.  $\Psi(X_{s_0}) = Y_{s_0}$).  Since $V_{s_0}$ is an open subset of $\overline{W}_{s_0}$ and $\overline{W}_{s_0}$ is irreducible,  $V_{s_0} \cap W_{s_0}$ is irreducible.  Consequently,  $X_{s_0}$ is irreducible and so is $Y_{s_0}$. Moreover,  according to Corollary~\ref{cor-1},  we have 
\begin{align*}
\dim \overline{Y}_{s_0} = \dim \overline{X}_{s_0} &= \dim \overline{V_{s_0} \cap W_{s_0}} + n_{s_0 + 1}  - r \\
&\ge n_1 + \cdots + n_{s_0+1} - \delta_{s_0} - (\delta_{s_0+1} - \delta_{s_0}) \\
&= n_1 + \cdots + n_{s_0+1} - \delta_{s_0 + 1}.
\end{align*}

Let $W_{s_0 + 1} \coloneqq \overline{Y}_{s_0} \cap ( \mathbb{K}^{n_1} \times \cdots \times \mathbb{K}^{n_{s_0} + 1} )$.  The proof is complete by verifying properties \ref{thm:Krull:eqi}--\ref{thm:Krull:eqiv} for $W_{s_0 + 1}$:
\begin{enumerate}[label = (\roman*)]
\item By the definition of $Y_{s_0}$,  $g_j$ must vanish on $Y_{s_0}$ and thus vanish on $W_{s_0 + 1}$ for each $j \in [N_2]$.  
\item We observe that all polynomials in the definition of $Y_{s_0}$ are homogeneous.  Thus,  for each $\lambda \in \mathbb{K}^\times$ and $\mathsf{w} \in Y_{s_0}$,  it holds that $\lambda \mathsf{w} \in Y_{s_0}$.  Since $\mathbb{K}$ is an infinite field,  this implies that the defining ideal of $Y_{s_0}$ is homogeneous.  In particular,  we may conclude that $\lambda W_{s_0+1} = W_{s_0 + 1}$ for any nonzero $\lambda \in \mathbb{K}$.
\item By definition,  we have $\overline{W}_{s_0 + 1} = \overline{Y}_{s_0}$.  The irreducibility of $\overline{W}_{s_0 + 1}$ follows from that of $Y_{s_0}$.
\item Since $\overline{W}_{s_0 + 1} = \overline{Y}_{s_0}$ and $\dim \overline{Y}_{s_0} \ge n_1 + \cdots + n_{s_0 + 1} - \delta_{s_0 + 1}$,  the desired inequality for $\dim \overline{W}_{s_0 + 1}$ follows immediately.  
\end{enumerate}
This completes the proof.
\end{proof}
Theorem~\ref{thm:Krull} immediately yields a lower bound on the dimension of the closures of fibers.
\begin{corollary}[Dimension of fibers]\label{cor:fiber dim}
Let $\mathbb{K}$ be an infinite field.  Suppose $G = (f_1,\dots, f_m):\mathbb{K}^{n_1} \times \cdots \times \mathbb{K}^{n_d} \to \mathbb{K}^m$ is a polynomial map such that $\mdeg f_1,\dots,  \mdeg f_m \le (1,\dots,  1)$.  Given $\mathsf{c} = (c_1,\dots,  c_m) \in \mathbb{K}^m$ such that $G^{-1}(\mathsf{c}) \ne \varnothing$,  we have $\dim Z \ge \sum_{i=1}^d n_i - 2^d m$ for each irreducible component $Z$ of $\overline{G^{-1}(c)}$.
\end{corollary}
Together with \eqref{eq:codim-ht},  Corollary~\ref{cor:fiber dim} implies that each irreducible component $W$ of $G_{\mathsf{c}}$ has dimension at least $\sum_{j \in [d}^d n_j - 2^d m$,  where $G_{\mathsf{c}}$ is the fiber of $G$ in the scheme-theoretic sense. By contrast,  the fiber dimension theorem \cite[Exercise~3.22]{hartshorne2013algebraic} shows that $\dim W \ge \sum_{j \in [d]} n_j  - m$.  Therefore,  Corollary~\ref{cor:fiber dim} is an analogue of the fiber dimension theorem for multilinear varieties.  As observed in Subsection~\ref{subsec:main},  the multilinearity in Theorem~\ref{thm:Krull},  and thus in Corollary~\ref{cor:fiber dim},  is essential.  In fact,  we have the following proposition,  recorded here for independent interest. 
\begin{proposition}
For any positive integer $n$,  there is a polynomial $P_n \in \mathbb{Q}[x_1,\dots,  x_n]$ such that $P_n^{-1}(c)$ is a finite set for every $c \in \mathbb{Q}$. 
\end{proposition}
\begin{proof}
The case where $n = 1$ is trivial.  Thus,  we suppose that $n \ge 2$.  We consider $P_2(x_1,x_2) = x_1^4 +  x_2^4$.  For $c \in \mathbb{Q}$,  $P_2^{-1} (c) \ne \varnothing$ only if $c \ge 0$.  If $c = 0$,  then $P_2^{-1} (0) = \{ (0,0) \}$.  If $c > 0$,  we consider the projective plane curve $C$ defined by the homogenization of $P_2 - c$:
\[
C \coloneqq \{ [X_1: X_2: t] \in \mathbb{P}^2: X_1^4 +  X_2^4 - c t^4 = 0\}.
\]
Since $c > 0$,  $C$ is a smooth quartic curve.  The degree-genus formula \cite[V.Example~1.5.1]{hartshorne2013algebraic} shows that $g(C) = \frac{(4-1)(4-2)}{2}=3>2$.  According to the Faltings' Theorem \cite{faltings1983endlichkeitssatze},  $C$ has only finitely rational points.  For $n > 2$,  we define $P_n$ inductively as 
\[
P_n(x_1,\dots,  x_n) \coloneqq P_2(P_{n-1}(x_1,\dots,  x_{n-1}),  x_n) = P_{n-1}(x_1,\dots,  x_{n-1})^4 + x_n^4.
\]
It is clear by construction that for any $c \in \mathbb{Q}$,  $P_n^{-1}(c)$ is a finite set.  
\end{proof}

\section{Applications}\label{sec:app}
For ease of reference,  we recall below the definition of various ranks of tensors and polynomials. 
\begin{definition}\label{def-ar}
Let $\mathbb{K}$ be a field.  For $f\in \Hom(\mathbb{K}^{n_1} \times \cdots \times \mathbb{K}^{n_d},\mathbb{K})$ and homogeneous polynomials $
P,  P_1,  \dots,  P_m\in \mathbb{K}[x_1,\dots,  x_n]_d$,  we have the following notions of ranks:
\begin{itemize}[label = $\diamond$]
\item The partition rank of $f$ is one,  written as $\PR_{\mathbb{K}} (f) = 1$,  if $f = g h$ for some $g\in  \Hom(\mathbb{K}^{n_{i_1}} \times \cdots \times \mathbb{K}^{n_{i_s}},\mathbb{K})$ and $g\in  \Hom(\mathbb{K}^{n_{i_{s+1}}} \times \cdots \times \mathbb{K}^{n_{i_d}},\mathbb{K})$ such that $\{i_1,\dots, i_d\} = [d]$.  In general,  the \emph{partition rank} of $f$ is
\[
\PR_{\mathbb{K}}(f) \coloneqq \min \Big\{ r \in \mathbb{N}:  f = \sum_{i=1}^r f_i,\;  \PR_{\mathbb{K}}(f_i) = 1 \Big\}.
\]
\item Let $F_f$ be the multilinear map in $\Hom(\mathbb{K}^{n_1}\times \cdots\times \mathbb{K}^{n_{d -1}}, \mathbb{K}^{n_d})$ corresponding to $f$ under the identification $\Hom(\mathbb{K}^{n_1}\times \cdots\times \mathbb{K}^{n_d}, \mathbb{K}) \simeq \Hom(\mathbb{K}^{n_1}\times \cdots\times \mathbb{K}^{n_{d-1}}, \mathbb{K}^{n_d})$.  The \emph{geometric rank} of $f$ is 
\[
\GR(f) \coloneqq \codim Z_{F_f}(\overline{\mathbb{K}}).
\]
\item Suppose that $\mathbb{K}$ is an infinite field.  The \emph{analytic rank} of $f$ is  
\[
    \AR_{\mathbb{K}}( f ) \coloneqq \codim \overline{Z_{F_f}(\mathbb{K})}.
\]
\item The strength of a homogeneous polynomial $P$ is one,  denoted by $\str(P) = 1$,  if $P = Q R$,  where $Q$ and $R$ are homogeneous polynomials of degree less than $k$.  In general,  the \emph{strength} of $P$ is 
\[
\str_{\mathbb{K}} (P) \coloneqq \min \Big\lbrace
r \in \mathbb{N}: P = \sum_{i = 1}^r P_i,\; \str_{\mathbb{K}} (P_i) = 1
\Big\rbrace,
\]
and the \emph{strength} of $(P_1,\dots,  P_m)$ is 
\[
\str_{\mathbb{K}}(P_1,\dots,  P_m) = \min \Big\lbrace
\str_{\mathbb{K}} (P): P \in \spa_{\mathbb{K}} \{P_1,\dots,  P_m\} \setminus \{0\}
\Big\rbrace.
\]
\item Let $Z(P_1,\dots,  P_m)_{\Sing} \coloneqq \{ \mathsf{v} \in \overline{\mathbb{K}}^n: \rank (\partial P_i/\partial x_j (\mathsf{v}) )_{i \in [m],  j\in [n]} < m \}$.  The \emph{Birch rank} of $(P_1,\dots,  P_m)$ is 
\[
\Brk(P_1,\dots,  P_m) \coloneqq \codim Z(P_1,\dots,  P_m)_{\Sing}.
\]
\end{itemize}
Here $Z_{F_f}(\overline{\mathbb{K}})$ and $Z_{F_f}(\mathbb{K})$ are multilinear varieties defined by $F_f$ over $\mathbb{K}$ and $\overline{\mathbb{K}}$,  respectively. 
\end{definition}
It is worth emphasizing that $\AR_{\mathbb{K}}(f)$ appears not to be canonically defined.  For instance,  $f$ also determines a multilinear map $F'_f\in \Hom(\mathbb{K}^{n_2} \times \cdots \times \mathbb{K}^{n_d},  \mathbb{K}^{n_1})$.  It is not obvious that $\codim \overline{Z_{F_f}(\mathbb{K})} = \codim \overline{Z_{F'_f}(\mathbb{K})}$. Surprisingly,  Proposition~\ref{prop-2} shows that $\AR_{\mathbb{K}} (f)$ does not depend on the way of slicing.

To proceed,  we need the following definition. Let $R_1,   R_2$ be functions defined on $\Hom(\mathbb{K}^{n_1} \times \cdots \times \mathbb{K}^{n_d}, \mathbb{K})$ for any positive integers $n_1,\dots,  n_d$.  
\begin{definition}[Linear equivalence]\label{def:equiv}
If there exist functions $c,  C: \mathbb{N} \to \mathbb{R}$ such that 
\begin{equation}\label{eq:equiv}
c(d) R_1(f ) \le R_2(f) \le C(d) R_1(f)
\end{equation}
for all $f  \in \Hom(\mathbb{K}^{n_1} \times \cdots \times \mathbb{K}^{n_d}, \mathbb{K})$ and positive integers $d,  n_1,\dots,  n_d$,  then we say that $R_1$ and $R_2$ are \emph{linearly equivalent},  and write $R_1 \asymp_{d} R_2$.  Moreover,  if $R_1$ and $R_2$ are defined for tuples $(f_1,\dots,  f_m)$ of multilinear functions,  then we write  $R_1 \asymp_{d,m} R_2$ if \eqref{eq:equiv} holds for some functions $c$ and $C$ of $d$ and $m$. If only the second inequality in \eqref{eq:equiv} holds,  then we write $R_2 = O_d(R_1)$.  We adopt the same notation for functions on $\mathbb{K}[x_1,\dots,  x_n]_d$.
\end{definition} 

\subsection{Ranks of tensors} In this subsection, we discuss the stability of tensor ranks and the linear equivalences among them. To this end, we establish several basic properties of the analytic rank. As a direct consequence of Theorem~\ref{thm:codim},  we obtain an alternative characterization of $\AR_{\mathbb{K}} (f)$.
\begin{corollary}[Alternative characterization]\label{cor:AR}
Let $\mathbb{K}$ be an infinite field. For any $f \in \Hom(\mathbb{K}^{n_1} \times \cdots \times \mathbb{K}^{n_d}, \mathbb{K})$,  we have 
\[
\AR_{\mathbb{K}}(f) = \min_{0 \le r \le \min\{n_{d-1},  n_d\}} \big\lbrace r + \codim \overline{W_{r,F_f} (\mathbb{K})}
\big\rbrace, 
\]
where $W_{r,F_f} (\mathbb{K}) \coloneqq \big\lbrace (\mathsf{v}_1,\dots,  \mathsf{v}_{d-2}) \in \mathbb{K}^{n_1} \times \cdots \times \mathbb{K}^{n_{d-2}}: \rank F_f(\mathsf{v}_1,\dots,  \mathsf{v}_{d-2},\cdot,  \cdot) = r \big\rbrace$.
\end{corollary}
By an argument similar to that in the proof of Theorem~3.2 in \cite{kopparty2020geometric}, Corollary~\ref{cor:AR} immediately implies:
\begin{proposition}[Independence on slicing]\label{prop-2}
For any infinite field $\mathbb{K}$,  $f \in \Hom(\mathbb{K}^{n_1} \times \cdots \times \mathbb{K}^{n_d}, \mathbb{K})$ and $j \in [d]$,  we have 
\[
\AR_{\mathbb{K}} (f) = \codim \overline{Z_{F_{f,j}}(\mathbb{K})},
\]
where $F_{f,j} \in \Hom(\mathbb{K}^{n_1} \times \cdots \mathbb{K}^{n_{j-1}} \times \mathbb{K}^{n_{j+1}} \times \mathbb{K}^{n_d},  \mathbb{K}^{n_j})$ is the multilinear map determined by $f$ via the isomorphism $\Hom(\mathbb{K}^{n_1}\times \cdots\times \mathbb{K}^{n_d}, \mathbb{K}) \simeq \Hom(\mathbb{K}^{n_1} \times \cdots \mathbb{K}^{n_{j-1}} \times \mathbb{K}^{n_{j+1}} \times \cdots \times \mathbb{K}^{n_d},  \mathbb{K}^{n_j})$.
\end{proposition}

Next,  we show that $\AR_{\mathbb{K}}$ is additive with respect to the direct sum,  and it is monotone under restriction.  
\begin{proposition}[Additivity and monotonicity] \label{prop-6}
Let $\mathbb{K}$ be an infinite field.  Given two multilinear functions $f\in \Hom(\mathbb{K}^{n_1} \times \cdots \times \mathbb{K}^{n_d},\mathbb{K})$ and $g \in \Hom(\mathbb{K}^{m_1} \times \cdots \times \mathbb{K}^{m_d},\mathbb{K})$,  we have 
\begin{enumerate}[label = (\alph*)]
\item $\AR_\mathbb{K}(f\oplus g)= \AR_\mathbb{K}(f)+\AR_\mathbb{K}(g)$.\label{prop-6:item1}
\item If there exists $\varphi_i \in \Hom(\mathbb{K}^{n_i},  \mathbb{K}^{m_i})$ for each $i \in [d]$ such that $f = g\circ (\varphi_1 \times \cdots \times \varphi_d)$,  then $ \AR_\mathbb{K}(f)\le \AR_\mathbb{K}(g)$.  \label{prop-6:item2}
\item If $m_1 = n_1$,  $\dots$,  $m_d = n_d$,  then $\AR_\mathbb{K}(f +  g) \le \AR_\mathbb{K}(f)+\AR_\mathbb{K}(g)$. \label{prop-6:item3}
\end{enumerate}
\end{proposition}
\begin{proof}
It is straightforward that \ref{prop-6:item1} holds by definition and \ref{prop-6:item3} can be deduced from \ref{prop-6:item1} and \ref{prop-6:item2}.  For \ref{prop-6:item2},  we note that by induction and Proposition~\ref{prop-2},  it suffices to prove for the case where $m_i = n_i$ and $\varphi_i = \Id_{\mathbb{K}^{n_i}}$ for $i \in [d-1]$,  which is clear from the definition of the analytic rank.
\end{proof}

We first prove Theorem~\ref{thm:linear}--\ref{thm:linear:item1},  which is an analogue of Conjecture~\ref{conj:pvsa}.
\begin{theorem}[Partition rank vs.  Analytic rank]\label{thm:pvsa}
Suppose $d \ge 2$.  For any infinite field $\mathbb{K}$ and multilinear function $f \in \Hom(\mathbb{K}^{n_1} \times \cdots \times \mathbb{K}^{n_d},\mathbb{K})$,  we have 
 \[
 \AR_\mathbb{K}(f)\le \PR_{\mathbb{K}} (f) \le (2^{d-1}-1)\AR_{\mathbb{K}}(f).\]
\end{theorem}
\begin{proof}
The right inequality is established in \cite[Theorem~A.1]{kazhdan2024schmidt}.  To prove the left inequality,  we proceed by induction on $d$.  The case of $d = 2$ is clear since $\AR_\mathbb{K}(f)= \PR_\mathbb{K}(f)=\rank(f)$.  Suppose that the inequality holds for any $d \le k-1$ for some integer $k \ge 3$.  We prove the inequality for $d = k$.   By Proposition~\ref{prop-6}, it is sufficient to prove $\AR_{\mathbb{K}} (f) \le 1$ if $\PR_{\mathbb{K}} (f) = 1$.  Without loss of generality, we may assume that $f(x_1,\dots,  x_k) = g (x_1,\dots,  x_r) h(x_{r+1},\dots,  x_k)$ for some integer $r \in [k]$,  $g \in \Hom(\mathbb{K}^{n_1} \times \cdots \times \mathbb{K}^{n_r},\mathbb{K})$ and $h \in \Hom(\mathbb{K}^{n_{r+1}} \times \cdots \times \mathbb{K}^{n_k},\mathbb{K})$.  If $r = 1$,  Proposition~\ref{prop-2} leads to $\AR_{\mathbb{K}}(f) = 1$.  If $r \ge 2$,  we consider $\widehat{h}$ in $\Hom(\mathbb{K}^{n_{r+1}} \times \cdots \mathbb{K}^{n_k} \times \mathbb{K},  \mathbb{K})$ defined by
\[
\widehat{h}(x_{r+1},\dots,  x_k,  y) = h(x_{r+1},\dots,  x_k) y.
\] 
Then $\PR_{\mathbb{K}}(\widehat{h}) = 1$ and the induction hypothesis implies that $\AR_{\mathbb{K}}( \widehat{h} ) \le 1$.  Let $F_{\widehat{h}}$ (resp.  $F_f$) be the multilinear map in $\Hom(\mathbb{K}^{n_{r+1}} \times \cdots \times \mathbb{K}^{n_{k-1}},  \mathbb{K}^{n_k})$ (resp.  $\Hom(\mathbb{K}^{n_1} \times \cdots \times \mathbb{K}^{n_{k-1}},  \mathbb{K}^{n_k})$) determined by $\widehat{h}$ (resp.  $f$). Then $\mathbb{K}^{n_1} \times \cdots \times \mathbb{K}^{n_r} \times Z_{F_{\widehat{h}}} \subseteq Z_{F_f} \subseteq \mathbb{K}^{n_1} \times \cdots \times \mathbb{K}^{n_k}$,  and this implies that 
\[
\AR_{\mathbb{K}}(f) = \codim \overline{Z}_{F_f} \le \codim \overline{Z}_{F_{\widehat{h}}} = \AR_{\mathbb{K}} (\widehat{h}) \le 1.  \qedhere
\]
\end{proof}
\begin{remark}
By Theorem~\ref{thm:pvsa},  we obtain 
\[
\AR_{\mathbb{K}} (f) \coloneqq  \min_{Z} \{ \codim Z \} \le \PR_{\mathbb{K}} (f) \le n_d,
\]
where $Z$ ranges over all irreducible components of $\overline{Z_{F_f}(\mathbb{K})}$ and $F_f$ is the multilinear map as in Definition~\ref{def-ar}.  For comparison,  we recall that Theorem~\ref{thm:Krull} yields 
\[
\max_{Z} \{ \codim Z \} \le 2^d n_d,
\] 
which provides an upper bound for the codimensions of all irreducible components of $\overline{Z_{F_f}(\mathbb{K})}$.
\end{remark} 
 
\begin{proposition}[Partition rank of direct sum]\label{prop-9}
Suppose that $\mathbb{K}$ is an infinite field and $d$ is an integer such that $d \ge 2$.  For each $f \in \Hom(\mathbb{K}^{n_1} \times \cdots \times \mathbb{K}^{n_d},  \mathbb{K})$ and positive integer $k$,  the following holds:
\[
\dfrac{k}{2^{d-1}-1}\PR_{\mathbb{K}}(f)\le \PR_{\mathbb{K}}(f^{\oplus k})\le k \PR_{\mathbb{K}}(f).
\]
\end{proposition}
\begin{proof}
The inequality $\PR_{\mathbb{K}}(f^{\oplus k})\le k\PR_{\mathbb{K}}(f)$ follows immediately from the definition of the partition rank.  Theorem~\ref{thm:pvsa} together with Proposition~\ref{prop-6} implies that
\[\PR_{\mathbb{K}}(f^{\oplus k})\ge \AR_{\mathbb{K}}(f^{\oplus k})=k\AR_{\mathbb{K}}(f)\ge \dfrac{k}{2^{d-1}-1}\PR_{\mathbb{K}}(f).  \qedhere
\]
\end{proof}
Next,  we prove Theorem~\ref{thm:stability}--\ref{thm:stability:item1}, resolving Conjecture~\ref{conj:stabilitypr} for perfect infinite fields.  By choosing bases,  $f$ may be identified with the same tensor $\mathsf{T}_{f} \in \mathbb{K}^{n_1} \otimes \cdots \otimes \mathbb{K}^{n_d} \subseteq \mathbb{F}^{n_1} \otimes \cdots \otimes \mathbb{F}^{n_d} $ for any field extension $\mathbb{F}/\mathbb{K}$.  Note that we must have $\PR_{\mathbb{F}}(f^{\mathbb{F}}) = \PR_{\mathbb{F}} (\mathsf{T}_f)$.
\begin{theorem}[Stability of partition rank]\label{thm:AKZ}
Let $\mathbb{K}$ be a perfect infinite field and let $d \ge 2$ be an integer.  For any $f \in \Hom(\mathbb{K}^{n_1} \times \cdots \times \mathbb{K}^{n_d},  \mathbb{K})$,  we have
\[
\PR_{\overline{\mathbb{K}}}(f) \le  \PR_{\mathbb{K}}(f)\le 6(2^{d-1}-1)\PR_{\overline{\mathbb{K}}}(f).
\]
\end{theorem}
\begin{proof}
The left inequality is proved in \cite[Lemma~5.1]{chen2024stability},  thus we only need to prove the right inequality.  Let $\mathbb{L}/\mathbb{K}$ be a finite extension of $\mathbb{K}$ such that $\mathsf{T}_f$ admits a decomposition $\mathsf{T}_f = \sum_{i=1}^r \mathsf{T}_i$,  where $r = \PR_{\overline{\mathbb{K}}} (\mathsf{T}_f)$,  $\mathsf{T}_i \in \mathbb{L}^{n_1} \otimes \cdots \otimes \mathbb{L}^{n_d}$ and $\PR_{\overline{\mathbb{K}}}(\mathsf{T}_i) = 1$.  It is clear that $\PR_{\mathbb{L}} (f) = \PR_{\mathbb{L}}  (\mathsf{T}_f) = r$.  Denote $m \coloneqq [\mathbb{L} : \mathbb{K}]$.  

By \cite[Proposition~3.1 \& Lemma~5.1]{chen2024stability},  we obtain 
\begin{equation}\label{thm:AKZ:eq1}
\PR_{\mathbb{K}}(f^{\oplus \lceil\frac{m-1}{2}\rceil}) \le \PR_{\mathbb{K}}(f)\le  (2[\mathbb{L}:\mathbb{K}]-1) \PR_{\mathbb{L}}(f).
\end{equation}
Moreover,  Proposition~\ref{prop-9} implies that 
\begin{equation}\label{thm:AKZ:eq2}
\frac{\lceil\frac{m-1}{2}\rceil}{2^{d-1}-1}\PR_{\mathbb{K}}(T) \le 
\PR_{\mathbb{K}}(f^{\oplus \lceil\frac{m-1}{2}\rceil}). 
\end{equation}
Combining \eqref{thm:AKZ:eq1} with \eqref{thm:AKZ:eq2},  we may derive 
\[
\PR_{\mathbb{K}}(f)\le \dfrac{(2^{d-1}-1)(2m-1)}{\lceil\frac{m-1}{2}\rceil}\PR_{\mathbb{L}}(f)\le 6(2^{d-1}-1)\PR_{\mathbb{L}}(f) = 6(2^{d-1}-1)r.  \qedhere
\]
\end{proof}
\begin{remark}
Theorem~\ref{thm:AKZ} follows from Proposition~\ref{prop-9},  which bounds $\PR_{\mathbb{K}}(f^{\oplus k})$ in terms of $\PR_{\mathbb{K}} (f)$.  By \cite{chen2024stability} and \cite{moshkovitz2024uniform},  the stability of partition rank over a finite field $\mathbb{F}_q$ is equivalent to
\begin{equation}\label{eq:asymp-directsum}
\PR_{\mathbb{F}_q} (f) = O_d \Big( \limsup_{k \to \infty} \frac{\PR_{\mathbb{F}_q} (f^{\oplus k})}{k} \Big).
\end{equation}
It is clear that \eqref{eq:asymp-directsum} holds whenever one has a bound similar to that in Proposition~\ref{prop-9} for finite fields. Therefore, it is reasonable to expect a proof of Conjecture~\ref{conj:stabilitypr} over finite fields by establishing an analogue of Proposition~\ref{prop-9}.
\end{remark}
Consequently,  we may establish Theorem~\ref{thm:linear}--\ref{thm:linear:item2} concerning the linear equivalence between partition rank and geometric rank.
\begin{theorem}[Partition rank vs. Geometric rank]\label{thm:pvsg}
Let $\mathbb{K}$ be a perfect infinite field and let $d \ge 2$ be an integer.  For any $f\in \Hom(\mathbb{K}^{n_1} \times \cdots \mathbb{K}^{n_d},  \mathbb{K})$,  we have 
\[
\GR(f) \le \PR_{\mathbb{K}}(f) \le 6(2^{d-1}-1)^{2}\GR(f).
    \]
\end{theorem}
\begin{proof}
It suffices to prove the right inequality,  since the left one is proved in \cite[Theorem~5]{kopparty2020geometric}.  According to Theorem~\ref{thm:pvsa},  it holds that $\PR_{\overline{\mathbb{K}}}(f) \le (2^{d-1} - 1) \AR_{\overline{\mathbb{K}}}(f)$.  We observe that $\AR_{\overline{\mathbb{K}}} (f) = \GR(f)$.  We conclude from Theorem~\ref{thm:AKZ} that
\[
\PR_{\mathbb{K}}(f) \le 6(2^{d-1}-1)\PR_{\overline{\mathbb{K}}}(f) \le 6(2^{d-1}-1)^2 \GR(f).  \qedhere
\]
\end{proof}

\subsection{Ranks of polynomials}
Suppose that $\ch(\mathbb{K}) = 0$ or $\ch(\mathbb{K}) > d$.  The \emph{polarization} of $P \in \mathbb{K}[x_1,\dots, x_n]_d$ is the unique symmetric multilinear function $f_P\in \Hom(\mathbb{K}^n \times \cdots \times \mathbb{K}^n,  \mathbb{K})$ satisfying $f_P(\mathsf{v},\dots,  \mathsf{v}) = P(\mathsf{v})$ for any $\mathsf{v} \in \mathbb{K}^n$.  The strength of $P$ is linearly equivalent to the partition rank of $f_P$,  as follows from the next lemma.
\begin{lemma}\cite[Claim~3.2]{LZ24}\label{lem:str-pr}
Let $\mathbb{K}$ be a field with $\ch(\mathbb{K}) = 0$ or $\ch(\mathbb{K}) > d \ge 2$.  For any $P \in \mathbb{K}[x_1,\dots,  x_n]_d$,  we have 
\[
\str_{\mathbb{K}} (P) \le \PR_{\mathbb{K}} (f_P) \le  \binom{d}{\lfloor d/2 \rfloor} \str_{\mathbb{K}} (P),
\]
where $f_P \in \Hom( \mathbb{K}^n \times \cdots \times \mathbb{K}^n,  \mathbb{K})$ is the polarization of $P$.
\end{lemma} 
By Lemma~\ref{lem:str-pr},  the stability of strength is an immediate consequence of that of partition rank.
\begin{lemma}[Stability of strength]\label{thm:str-stability}
Suppose that $\ch(\mathbb{K}) = 0$ or $\mathbb{K}$ is infinite with $\ch(\mathbb{K}) > d \ge 2$.  For any $P \in \mathbb{K}[x_1,\dots,  x_n]_d$,  we have 
\[
\str_{\overline{\mathbb{K}}} (P)
 \le \str_{\mathbb{K}} (P) \le 6(2^{d-1} - 1) \binom{d}{\lfloor d/2 \rfloor} \str_{\overline{\mathbb{K}}} (P).
\]
\end{lemma}
\begin{proof}
The left inequality trivially holds. According to Lemma~\ref{lem:str-pr},  we have 
\[
\str_{\mathbb{K}} (P) \le \PR_{\mathbb{K}} (f_P),\quad  \PR_{\overline{\mathbb{K}}} (f_P) \le  \binom{d}{\lfloor d/2 \rfloor} \str_{\overline{\mathbb{K}}} (P).
\]
Combining this with Theorem~\ref{thm:AKZ},  we obtain
\[
\str_{\mathbb{K}}(P) \le \PR_{\mathbb{K}}(f_P) \le 6(2^{d-1} - 1) \PR_{\overline{\mathbb{K}}} (f_P) \le 6(2^{d-1} - 1) \binom{d}{\lfloor d/2 \rfloor} \str_{\overline{\mathbb{K}}} (P).   \qedhere
\]
\end{proof}
Next,  we show that the strength is linearly equivalent to the Birch rank.  To this end,  we recall the following result on the linear equivalence between strength and Birch rank over algebraically closed fields.
\begin{lemma}\cite[Theorem~1.3]{kazhdan2024schmidt}\label{lem:brk-str}
Let $\mathbb{K}$ be an algebraically closed field with $\ch(\mathbb{K}) = 0$ or $\ch(\mathbb{K}) > d \ge 2$.  For any homogeneous polynomial $P \in \mathbb{K}[x_1,\dots,  x_n]$ of degree $d$,  the following holds: 
\[
\frac{\Brk (P)}{2} \le  \str_{\mathbb{K}} (P) \le (d-1) \Brk(P).
\]
\end{lemma}
The desired linear equivalence between strength and Birch rank is obtained by combining Lemmas~\ref{thm:str-stability} and~\ref{lem:brk-str}.
\begin{lemma}[Strength vs.  Birch rank]\label{lem:str-birch}
Suppose that $\ch(\mathbb{K}) = 0$ or $\mathbb{K}$ is infinite with $\ch(\mathbb{K}) > d \ge 2$.  For any homogeneous polynomial $P \in \mathbb{K}[x_1,\dots,  x_n]$ of degree $d$,  we have 
\[
\frac{\Brk (P)}{2} \le  \str_{\mathbb{K}}(P) \le 6(2^{d-1} - 1) \binom{d}{\lfloor d/2 \rfloor}(d-1) \Brk(P).
\]
\end{lemma}
\begin{proof}
Lemma~\ref{lem:brk-str} implies 
\[
\frac{\Brk (P)}{2} \le  \str_{\overline{\mathbb{K}}} (P) \le (d-1) \Brk(P).
\]
Thus,  Lemma~\ref{thm:str-stability} leads to the following: 
\[
\frac{\Brk (P)}{2} \le  \str_{\overline{\mathbb{K}}}(P)  \le \str_{\mathbb{K}}(P) \le 6(2^{d-1} - 1) \binom{d}{\lfloor d/2 \rfloor} \str_{\overline{\mathbb{K}}} (P) \le 6(2^{d-1} - 1) \binom{d}{\lfloor d/2 \rfloor}(d-1) \Brk(P).
\]
\end{proof}
Now we are ready to prove Theorem~\ref{thm:linear}--\ref{thm:linear:item3}.
\begin{theorem}[Collective strength vs.  collective Birch rank]\label{thm:collective str-birc}
Let $d \ge 2$ be an integer.  Suppose that either $\ch(\mathbb{K}) = 0$ or $\mathbb{K}$ is perfect and infinite with $\ch(\mathbb{K}) > d$.  For any $P_1,\dots,  P_m \in \mathbb{K}[x_1,\dots,  x_n]_d$,  we have 
\[
\frac{\Brk(P_1,\dots,  P_m)}{2} \le \str_{\mathbb{K}} (P_1,\dots,  P_m)  \le 6(2^{d-1} - 1) \binom{d}{\lfloor d/2 \rfloor}(d-1)m ( \Brk(P_1,\dots,  P_m) + m - 1).
\]
\end{theorem}
\begin{proof}
The right inequality follows from the same argument in the proof of \cite[Lemma~2.12]{baily2024strength},  by replacing Theorem~1.4 there with Lemma~\ref{lem:str-birch}.  For the left inequality,  we choose $a_1,\dots,  a_m \in \mathbb{K}$ such that 
\[
\str_{\mathbb{K}} (P) = \str_{\mathbb{K}} (P_1,\dots,  P_m),\quad P \coloneqq a_1 P_1 + \cdots + a_m P_m \in \mathbb{K}[x_1,\dots,x_n]_d.
\]
We notice that if $\mathsf{v} \in Z(P)_{\Sing}$,  then it holds that $\mathsf{v} \in Z(P_1,\dots,  P_m)_{\Sing}$ since
\[
\begin{bmatrix}
\partial P_1 / \partial x_1(\mathsf{v}) & \cdots & \partial P_m / \partial x_1(\mathsf{v}) \\ 
\vdots & \ddots & \vdots \\
\partial P_1 / \partial x_n(\mathsf{v}) & \cdots & \partial P_m / \partial x_n (\mathsf{v})
\end{bmatrix} \begin{bmatrix}
a_1 \\
\vdots \\
a_m
\end{bmatrix} = 0.
\]
This together with Lemma~\ref{lem:brk-str} implies 
\[
\str_{\mathbb{K}} (P) \ge 
\str_{\overline{\mathbb{K}}} (P) \ge \frac{\Brk (P)}{2} \ge \frac{\Brk (P_1,\dots,  P_m)}{2}. \qedhere
 \]
\end{proof}
The following corollary of Theorem~\ref{thm:collective str-birc} is straightforward.
\begin{corollary}\label{cor:str-to-Birch}
Suppose that either $\ch(\mathbb{K}) = 0$ or $\mathbb{K}$ is perfect and infinite with $\ch(\mathbb{K}) > d \ge 2$.  Given $P_1,\dots,  P_m \in \mathbb{K}[x_1,\dots,  x_n]_d$ such that $\str_{\mathbb{K}} (P_1,\dots,  P_m) \ge 6(2^{d-1} - 1) \binom{d}{\lfloor d/2 \rfloor}(d-1)m (r + m - 1)$ for some $r \in [n]$,  it holds that $\Brk(P_1,\dots,  P_m) \ge r$.
\end{corollary}
Lastly, we turn to the proof of Theorem~\ref{thm:stability}--\ref{thm:stability:item2}, which settles Conjecture~\ref{conj:stabilitystr} for all infinite perfect fields.  
\begin{theorem}[Stability of collective strength]\label{thm:stability-collective-str}
Let $d \ge 2$ be an integer.  Suppose that either $\ch(\mathbb{K}) = 0$ or $\mathbb{K}$ is perfect and infinite with $\ch(\mathbb{K}) > d$.  For any $P_1,\dots,  P_m \in \mathbb{K}[x_1,\dots,  x_n]_d$,  the following inequalities hold: 
\[
\str_{\overline{\mathbb{K}}} (P_1,\dots,  P_m)  \le \str_{\mathbb{K}} (P_1,\dots,  P_m) \le 6(2^{d-1} - 1) \binom{d}{\lfloor d/2 \rfloor}(d-1)m (2 \str_{\overline{\mathbb{K}}} (P_1,\dots,  P_m) + m - 1).
\]
\end{theorem}
\begin{proof}
The left inequality is trivial.  Denote $C \coloneqq 6(2^{d-1} - 1) \binom{d}{\lfloor d/2 \rfloor}(d-1)m$. By Theorem~\ref{thm:collective str-birc},  we derive  
\[
\str_{\mathbb{K}} (P_1,\dots,  P_m) \le C (\Brk(P_1,\dots,  P_m) + m - 1) \le C (2 \str_{\overline{\mathbb{K}}} (P_1,\dots,  P_m) + m - 1),
\]
and this completes the proof.  
\end{proof}

\subsection{Number of integral solutions} 
Let $\mathbb{K}$ be a number field and let $d \ge 2$ be an integer.  By Corollary~\ref{cor:str-to-Birch},  the same argument as in the proof of \cite[Theorem~2.5]{LZ24} yields a Schmidt-type result for the number of integral solutions of a polynomial system.  

To state the result,  we introduce some notations.  Let $\mathcal{O}$ be the ring of integers in $\mathbb{K}$.  Suppose that for each $s \in [d]$,  $g_{s,1},\dots,  g_{s, m_s} \in \mathcal{O}[x_1,\dots,  x_n]$ are polynomials of degree $s$,  where $m_s \ge 0$ and $m_d \ge 1$.  Given an integral ideal $\mathfrak{a}$ of $\mathcal{O}$,  we choose an integral basis $\omega_1$,  $\dots$,  $\omega_\ell$ of $\mathfrak{a}$.  Since $\omega_1$,  $\dots$,  $\omega_\ell$ is also an $\mathbb{R}$-basis of $\mathbb{K} \otimes_{\mathbb{Q}} \mathbb{R}$,  we may identify $\mathbb{K} \otimes_{\mathbb{Q}} \mathbb{R}$ with $\mathbb{R}^{\ell}$.  A subset $B \subseteq (\mathbb{K} \otimes_{\mathbb{Q}} \mathbb{R})^n$ is a \emph{box with respect to this basis}, if its image $B' \subseteq \mathbb{R}^{\ell n}$ under the identification $(\mathbb{K} \otimes_{\mathbb{Q}} \mathbb{R})^n \simeq \mathbb{R}^{\ell n}$ is of the form $B' = [a_1,b_1] \times \cdots [a_{\ell n},  b_{\ell n}]$.  We define for a box $B$ the \emph{counting function} $N_B: \mathbb{R} \to \mathbb{N} \cup \{0\}$ by
\[
N_B(t) \coloneqq \big\lvert \{
\mathsf{v} \in \mathfrak{a}^n \cap (t B): g_{s,j} (\mathsf{v}) = 0,\; s\in [d],  \; j \in [m_s]
\} \big\rvert.
\]
We denote $M_0 \coloneqq 0$ and for each $s \in [d]$,  we write 
\[
M_s \coloneqq \sum_{j=1}^s j m_j,\quad 
u_s \coloneqq \sum_{j=s}^{d} 2^{j-1}(j-1)m_j,\quad \lambda_s \coloneqq M_s (2^{s-1} + u_{s + 1}) + u_{s+1} + \sum_{j=s+1}^d u_j m_j.
\] 
Lastly,  we define $\Lambda \coloneqq \max \{\lambda_s: s = 0 \text{~or~} m_s \ge 1\}$,  where $\lambda_0  = u_1 + \sum_{j=1}^d u_j m_j$. 
\begin{proposition}[Schmidt-type result for integral solutions]\label{prop:integral}
Suppose that $\mathbb{K}$ is a number field,  $d \ge 2$ is an integer, and for each $s \in [d]$,  $g_{s,1},\dots,  g_{s, m_s} \in \mathcal{O}[x_1,\dots,  x_n]$ are polynomials of degree $s$,  where $m_s \ge 0$ and $m_d \ge 1$.  Let $\mathfrak{a}\subseteq \mathcal{O}$ be an integral ideal with a $\mathbb{Z}$-basis $\omega_1,\dots,  \omega_{\ell}$,  and let $B \subseteq \mathbb{K} \otimes_{\mathbb{Q}} \mathbb{R}$ be a box.  Denote by $P_{s,j} \in \mathcal{O}[x_1,\dots, x_n]_{s}$ the highest degree part of $g_{s,j}$ for each $s \in [d]$ and $j \in [m_s]$.  If
\[
\str_{\mathbb{K}}(P_{s,1},\dots,  P_{s,m_s}) > 6(2^{s-1} - 1)\binom{s}{\lfloor s/2 \rfloor}(s-1) m_s (m_d \Lambda + m_s - 1)
\]
for any $s \in [d]$ with $m_s \ge 1$,  then there is a constant $\delta \coloneqq \delta(\mathbb{K},  P_{s,j}) > 0$ such that for a sufficiently large $t$,  it holds that
\begin{equation}\label{prop:integral:eq1}
N_B(t) = \mu t^{\ell (n - M_d)} + O(t^{\ell (n - M_d) - \delta}),
\end{equation}
where both $\mu$ and the lower order term in \eqref{prop:integral:eq1} only depend on $\mathbb{K}$,  $\mathfrak{a}$,  $\omega_1,\dots,  \omega_{\ell}$ and $g_{s,j}$'s.
\end{proposition}
\begin{remark}
The leading coefficient $\mu$ in \eqref{prop:integral:eq1} equals $\mathfrak{S} \mathfrak{I}$,  where $\mathfrak{S}$ (resp.  $\mathfrak{I}$) is the singular series (resp.  singular integral).  We refer interested readers to \cite[Subsection~1.3]{FM17} for precise definitions of $\mathfrak{S}$ and $\mathfrak{I}$.  Moreover,  Proposition~\ref{prop:integral} is an effective version of \cite[Theorem~2.5]{LZ24}. On the other side,  given $P_1,\dots,  P_m \in \mathbb{K}[x_1,\dots,  x_n]_d$,  it is proved in \cite[Proposition~2.13]{baily2024strength} that \eqref{prop:integral:eq1} holds if
\[
\str_{\mathbb{K}}(P_1,\dots,  P_m) > (2^{d-1} - 1) \binom{d}{\lfloor d/2 \rfloor} (d-1)^2 m \left[  m(m + 1)(d - 1)2^{d-1} + m - 1 \right].
\]
We note that in this case,  Proposition~\ref{prop:integral} slightly improves the above lower bound,  by a factor of $d-1$ when $d \ge 8$:
\[
\str_{\mathbb{K}}(P_1,\dots,  P_m) > 6(2^{d-1} - 1) \binom{d}{\lfloor d/2 \rfloor} (d-1) m (  m^2 d 2^{d-1} + m - 1 ).
\] 
\end{remark}

\bibliographystyle{abbrv}
\bibliography{ref}
\end{document}